\documentclass[reqno,11pt]{amsart}
\usepackage{amsthm,amsfonts,amssymb,euscript,mathrsfs,graphics,color,amsmath,latexsym,marginnote,hyperref}
\usepackage[latin1]{inputenc}
\usepackage{amsmath}
\usepackage{graphicx}
\usepackage{amsfonts}
\usepackage{amssymb}
\usepackage{amsthm}
\usepackage{dsfont}    
\usepackage{mathrsfs}
\usepackage{amsthm}
\usepackage{verbatim}
\usepackage{esint}
\usepackage{stmaryrd}

\usepackage{enumitem}

\makeatletter
\renewcommand{\subsection}{%
	\@startsection
	{subsection}{2}{\z@}%
	{-3.25ex \@plus -1ex \@minus -.2ex}%
	{1.5ex \@plus .2ex}%
	{\normalfont\bfseries}}
\makeatother
\theoremstyle{plain}

\usepackage{hyperref}
\usepackage{times}
\usepackage{todonotes}
\usepackage{mathtools}

\usepackage{marginnote}

\numberwithin{equation}{section}

\newtheorem{theorem}{Theorem}[section]
\newtheorem{proposition}[theorem]{Proposition}
\newtheorem{lemma}[theorem]{Lemma}
\newtheorem{corollary}[theorem]{Corollary}

\newtheorem{remark}[theorem]{Remark}
\newtheorem{remarks}[theorem]{Remark}

\newcommand{\R}{\mathbb{R}}
\newcommand{\N}{\mathbb{N}}
\newcommand{\Z}{\mathbb{Z}}

\newcommand{\C}{\mathbb{C}}
\newcommand{\T}{\mathbb{T}}
\newcommand{\M}{\mathcal{M}}
\renewcommand{\L}{\mathcal{L}}
\newcommand{\A}{\mathcal{A}}
\newcommand{\K}{\mathcal{K}}
\newcommand{\B}{\mathcal{B}}
\newcommand{\D}{\mathcal{D}}

\newcommand{\Ff}{\mathcal{F}}
\renewcommand{\P}{\mathcal{P}}

\newcommand{\Cc}{\mathcal{C}}
\newcommand{\E}{\mathcal{E}}

\DeclarePairedDelimiter\jap{\langle}{\rangle}

\renewcommand{\phi}{\varphi}

\newcommand{\dt}{\partial_t}
\newcommand{\dtt}{\partial_{tt}}
\newcommand{\F}{\mathcal{F}}

\newcommand{\Id}{{\rm Id}}

\newcommand{\eb}{\bar{\eta}}

\newcommand{\w}{\bar{w}}

\newcommand{\z}{\bar{z}}

\newcommand{\xt}{X_{{\bf Res},3}}

\renewcommand{\epsilon}{\varepsilon}
\newcommand{\y}{\texttt{y}}
\newcommand{\Mat}{\texttt{M}}
\renewcommand{\Re}{{\rm Re}}

\renewcommand{\S}{\mathcal{S}}
\newcommand{\MM}{\mathfrak{M}}
\newcommand{\q}{\mathcal{Q}}
\newcommand{\dc}{\delta_0}
\renewcommand{\c}{\mathfrak{c}}
\newcommand{\rh}{{\bf \varrho}}
\newcommand{\qb}{\bar{q}}

\newcommand{\f}{{\bf f}}

\renewcommand{\epsilon}{\varepsilon}

\providecommand{\vect}[2]{{\bigl(\begin{smallmatrix}#1\\#2\end{smallmatrix}\bigr)}}

\newenvironment{system}
{\left\lbrace\begin{array}{@{}l@{}}}
	{\end{array}\right.\kern-\nulldelimiterspace}

\makeatletter

\def\l@subsection{\@tocline{2}{0pt}{2.5pc}{5pc}{}}
\def\l@subsubsection{\@tocline{3}{0pt}{4.5pc}{5pc}{}}
\renewcommand\tocchapter[3]{%
	\indentlabel{\@ifnotempty{#2}{\ignorespaces#2.\quad}}#3%
}
\def\l@subsection{\@tocline{2}{0pt}{2.5pc}{5pc}{}}

\begin{document}
	\title{Quasilinear normal form for the Kirchhoff-Poho{\v z}aev equation}
	\date{}
 \author{Emanuele Haus}
 \address{\scriptsize{Dipartimento di Matematica e Fisica, Universit\`a degli Studi Roma Tre, 
 		Largo San Leonardo Murialdo 1, 00146 Roma}}
 \email{emanuele.haus@uniroma3.it}
\author{Simone Marrocco}
\address{\scriptsize{Dipartimento di Matematica e Fisica, Universit\`a degli Studi Roma Tre, 
		Largo San Leonardo Murialdo 1, 00146 Roma}}
  \email{simone.marrocco@uniroma3.it}

	\maketitle
	\begin{abstract}
		On the $n$-dimensional torus $\mathbb{T}^n$, we consider the Kirchhoff-Poho{\v z}aev equation, which is the only Kirchhoff-type equation known to admit global 
		solutions for small initial data in $H^s \times H^{s-1}$, $s \geq 2$. 
		We study this equation from a dynamical perspective by means of a quasilinear 
		normal form approach. We perform two steps of normal form reduction and show 
		that the resulting cubic and quintic terms do not contribute to the Sobolev 
		energy estimates. This cancellation reveals a special algebraic structure of the 
		equation and represents a first step towards a dynamical explanation of its 
		exceptional global well-posedness. As a further consequence, we obtain improved 
		bounds on the growth of Sobolev norms and improved lower bounds on the existence 
		time for initial data in $H^s \times H^{s-1}$ with $s \in [3/2, 2)$.
		\end{abstract} 

	\tableofcontents 

	\section{Introduction}	
	 In this paper, we consider the Cauchy problem associated with the Kirchhoff-Poho{\v z}aev equation 
	 \begin{equation}\label{MK}
	 	\dtt u-\left(\dfrac{1}{1+c\int_{\T^n}\left|\nabla u\right|^2dx}\right)^2\Delta u=0,\quad c\in\R
	 \end{equation} on the $n$-dimensional torus $\T^n$ and with initial data $u_0:=u(0,\cdot)$, $v_0:=u_t(0,\cdot)$ in Sobolev class, namely $(u_0,v_0)\in H^{s}(\T^n,\R)\times  H^{s-1}(\T^n,\R)$. As we will discuss later, we are going to need at least $s\geq3/2$ because of the local well-posedness theory.
	 
	Equation \eqref{MK} was first introduced in the work of Poho{\v z}aev \cite{P2} (see also the more recent note \cite{BoitiManfrin2023}) and represents, as far as we know, the only non-trivial example of Kirchhoff-type equation for which global-in-time existence is known to hold for all initial data. It is useful to briefly recall here why equation \eqref{MK} is special. Consider a general Kirchhoff-type equation, of the form
	\begin{equation}\label{genk}
	\partial_{tt} u -{\bf f}\!\left(\int_{\mathbb T^n} |\nabla u|^2\,dx\right)\,\Delta u=0,
	\end{equation}
	where ${\bf f}: [0,+\infty) \to (0,+\infty)$ is a given smooth function. Equation \eqref{MK} corresponds to the special case ${\bf f}(y)=(1+cy)^{-2}$, while the standard Kirchhoff equation corresponds to ${\bf f}(y)=1+y$. Then, as we will discuss more in detail later, equation \eqref{genk} is known to be locally well-posed for $(u,\partial_t u) \in H^s \times H^{s-1}$ for $s\geq\frac32$. For a generic choice of $\bf f$, equation \eqref{genk} has no known conserved quantity that controls the norm $H^{\frac32}\times H^{\frac12}$ of the local well-posedness, thus one cannot deduce that solutions are global in time. However, one may consider the quantity
	\begin{equation}\label{def Kf}
			K_{\bf f}(u):=\dfrac{\int_{\T^n}\left|\nabla \dt u\right|^2dx}{\sqrt{\f\left(\| \nabla u \|_{L^2}^2\right)}}+\sqrt{\f\left(\| \nabla u \|_{L^2}^2\right)}\int_{\T^n}(\Delta u)^2dx
		+\dfrac{\f'\left(\| \nabla u \|_{L^2}^2\right)}{2\left(\f\left(\| \nabla u \|_{L^2}^2\right)\right)^{\frac32}}\left(\int_{\T^n}\nabla\dt u\cdot\nabla u\,dx\right)^2,
	\end{equation}
	first introduced by Poho{\v z}aev (see \cite{P}, \cite{P2}), which gives control over the $H^2 \times H^1$ norm of the solution. This implies that if, for some choice of $\bf f$, the quantity $K_{\bf f}$ is a conservation law, then the solutions of equation \eqref{genk} with initial data in $H^2\times H^1$ are global in time for that choice of $\bf f$. By a straightforward computation, involving several cancellations, one gets
\begin{equation}\label{ema33}
		\dfrac{d}{dt} K_{\bf f}(u(t))= {\bf g}\left(\| \nabla u \|_{L^2}^2\right) \left(\int_{\T^n}\nabla\dt u\cdot\nabla u\,dx\right)^3,
		\qquad\text{where}\quad
		{\bf g}(y):= \frac{d}{dy}\frac{{\bf f}'(y)}{\left({\bf f}(y)\right)^{\frac32}}.
		\end{equation}
	Thus, requiring that $K_{\bf f}$ is a conservation law amounts to the condition $\bf g\equiv0$, which is an elementary ODE in the function $\bf f$ and gives ${\bf f}(y)=(a+by)^{-2}$, where $a,b$ are real constants with $(a,b)\neq(0,0)$. The case $b=0$ is not interesting because \eqref{genk} would become a linear wave equation, while the case $a=0$ leads to a singular equation not defined at $u\equiv0$. If $a\neq0$, then one can set $a=1$ by rescaling time, thus obtaining equation \eqref{MK}. One could also rescale the size of the solution $u$ to get rid of the constant $c$ in \eqref{MK} and set $c=\pm1$, but the sign of $c$ cannot be changed, so we choose not to do this further scaling.
	
	Very recently, Boiti and Manfrin have proved (see \cite{BoitiManfrin2025, BoitiManfrin2026}) that equation \eqref{MK} has a whole hierarchy of infinitely many conservation laws, controlling Sobolev norms of arbitrarily high orders of the solution. This, of course, raises questions whether equation \eqref{MK} may possess an integrable structure, which might also be the reason behind the property that solutions of \eqref{MK} are global in time. Naturally, this also triggers the question whether, similarly as for the one-dimensional cubic NLS and KdV equations, the normal form of equation \eqref{MK} may be formally integrable at all orders, at least in dimension $n=1$.
	
	\bigskip
		
The aim of this paper is to start the investigation in this direction. Our main result is that, after the first two steps of quasilinear normal form for equation \eqref{MK}, all the non-resonant terms that might lead to growth of Sobolev norms are erased from the equation. In dimension $n=1$, this also has a consequence on the lifespan of solutions that do not belong to $H^2\times H^1$, for which global well-posedness is not known. Indeed, as a consequence of our normal form result, we get that all the solutions corresponding to initial data of size $\varepsilon$ in the Sobolev space $H^{3/2}\times H^{1/2}$ have a time of existence that is at least of order $\varepsilon^{-6}$. This should be compared with the result in \cite{BaldiHaus1}, where the authors prove that the for the standard Kirchhoff equation one can erase the ``bad'' non-resonant terms from the equation after the first step of quasilinear normal form, thus deducing a time of existence of order $\varepsilon^{-4}$. Differently from the Kirchhoff-Poho{\v z}aev equation, for the standard Kirchhoff equation such a cancellation fails to occur after the second step of normal form, as proved in \cite{BaldiHaus2}.

	The method that we use here and that was used in \cite{BaldiHaus1,BaldiHaus2} is deeply inspired by the work of Delort in \cite{Delort2009LONGTIMESS}, \cite{Delort1} in which he constructs a normal form for quasilinear Klein-Gordon equations on the circle.
	

\bigskip
	\subsection{Main results}
Since the normal form analysis is naturally carried out within the
Hamiltonian framework, we reformulate equation \eqref{MK} as an
infinite-dimensional Hamiltonian system. Setting $v:=\partial_t u$,
equation \eqref{MK} is equivalent to
\begin{equation}\label{system1}
	\begin{cases}
		\dt u=v\\
		\dt v=\left(\dfrac{1}{1+c\int_{\T^n}\left|\nabla u\right|^2dx}\right)^2\Delta u
	\end{cases}.
\end{equation}
Throughout this section we adopt the notation that will be introduced later, in Section
\ref{functional setting}. In particular, $H^s_0(\T^n,\R)$ denotes the Sobolev
space of real-valued, zero-mean functions, $\Lambda:=\sqrt{-\Delta}$, and
$H^s_0(\T^n,\mathrm{c.c.})$ denotes the space of pairs $(z,\bar z)$ of
complex conjugate, zero-mean functions with $z\in H^s_0(\T^n,\C)$,
endowed with the norm $\|(z,\bar z)\|_s:=\|z\|_s$.\\
Finally, we set
\begin{equation}\label{m0}
	m_0:=\begin{cases}
		1 & \text{if } n=1,\\[2pt]
		\tfrac32 & \text{if } n\geq2.
	\end{cases}
\end{equation}\\\\
The main result of this paper is the following:
	\begin{theorem}\label{Theorem1}
		For every $s\geq m_0$ there exist $r>0$ and a bounded, injective transformation
		$$\Phi:B_r\left(H^s_0(\T^n,\mathrm{c.c.})\right)\rightarrow H_0^{s+\frac12}(\T^n,\R)\times H_0^{s-\frac12}(\T^n,\R)$$
		that conjugates system \eqref{system1} to one of the form
		$$\dt\begin{pmatrix}
			q\\
			\qb\end{pmatrix}=\mathcal{X}(q,\qb)=\D_1(q,\qb)+\mathcal{Z}(q,\qb)+\mathcal{R}_{\geq7}(q,\qb),$$
		where $\D_1$ is linear, while $\mathcal{Z}$ contains only terms of homogeneity at least three and commute with $\D_1$. More precisely,
		\begin{equation}\label{Zsplitting}
		\mathcal{Z}=\mathcal{F}\,\mathcal{D}_1+(1+\mathcal{F})\,\mathcal{Z}_3+\mathcal{Z}_5,
		\end{equation}
		where $\mathcal{F}$ is a real scalar function of homogeneity at least two, and $\mathcal{Z}_3$, $\mathcal{Z}_5$ are homogeneous of degree three and five,
		respectively.\\Moreover, $\D_1$ and $\mathcal{Z}$ do not contribute to the energy estimates, namely the Sobolev norms of the solutions of the truncated system
		\begin{equation}\label{truncated}
		\dt(q,\qb)=\D_1(q,\qb)+\mathcal{Z}(q,\qb)
		\end{equation}
		are constant. Finally, the vector field $\mathcal{R}_{\geq7}$  contains only terms of homogeneity $\geq7$ and is bounded.
	\end{theorem}
	\begin{remark}
		We emphasize that the boundedness of $\mathcal{R}_{\geq7}$ on $H^s$ is crucial for performing the energy estimates on the solutions.
	\end{remark}
	\begin{remark}
		The variation of the regularity threshold $m_0$ in \eqref{m0} with respect to the dimension $n$ is linked to the fact that the coefficients of the map $\Phi$ have denominators of the form $|j|-|k|$, for $j,\,k\in\Z^n\backslash\{0\}$ (see Remark \ref{bh1}). While in the one-dimensional case this difference is always at least 1, for $n\geq2$ it can accumulate near zero with a lower bound of the form $||j|-|k||\geq1/(|j|+|k|)$ and this causes the additional loss of half a derivative.\\
		These denominators arise from the first step of normal form. Remarkably, the
		second step introduces no denominator of the form
		$|j|+|k|-|\ell|$, which would otherwise cause a further loss of
		half a derivative when $n\geq2$; as we explain below, this is a consequence of a
		cancellation specific to the Kirchhoff-Poho\v{z}aev equation.
	\end{remark}
	\begin{remark}
		As anticipated before, the proof of Theorem \ref{Theorem1} relies on the techniques developed in \cite{BaldiHaus1}, \cite{BaldiHaus2} for the standard Kirchhoff equation with ${\bf f}(y)=1+y$. In these works, however, the integrable behavior appears only for the resonant cubic terms, while it fails for the quintic ones.
	\end{remark}
	Theorem \ref{Theorem1} yields two dynamical consequences.\\

\medskip
\noindent{\bf Improved lifespan for $n=1$:}\\
		The works of Dickey \cite{Dickey}, Arosio-Panizzi \cite{ArosioPanizzi} and Poho{\v z}aev \cite{P} guarantee the existence of the solutions of \eqref{MK} for initial data in $H^{s}\times H^{s-1}$ with $s\geq\frac{3}{2}$ and provide a lower bound on the time of existence of order $\epsilon^{-2}$, where $\epsilon$ denotes the size of the initial data.\\
		The work of Baldi-Haus \cite{BaldiHaus1}, which can be adapted with no substantial differences to the Kirchhoff-Poho{\v z}aev equation, provides a lower bound on the time of existence of order $\varepsilon^{-4}$ for initial data in $H^{s}\times H^{s-1}$ with $s\geq m_0+1/2$, i.e.\ for $s\geq3/2$ in dimension $n=1$ and for $s\geq2$ in dimension $n\geq2$.\\
		On the other hand, the theorem of Poho{\v z}aev \cite{P2} ensures existence for all times only above the regularity threshold of $s=2$.\\
		In this regard, Theorem \ref{Theorem1} implies an improved lower bound in dimension $n=1$ and in the case $s\in[3/2,2)$:
		\begin{theorem}\label{theorem!}
			There exist constants $\epsilon_0>0,\,C>0$ such that the following holds:\\
			if $(u_0,v_0)\in H^{s}(\T,\R)\times H^{s-1}(\T,\R)$ with $s\in\big[\frac{3}{2},2\big)$ and 
			$\|u_0\|_{s}+\|v_0\|_{s-1}\leq\epsilon_0,$
			equation \eqref{MK} admits a unique solution $u\in C^0([0,T],H^s(\T,\R))\cap C^1([0,T],H^{s-1}(\T,\R))$ with 
			$$T\sim(\|u_0\|_{s}+\|v_0\|_{s-1})^{-6}$$
			and
			\begin{equation}\label{miserve11}
				\max_{t\in[0,T]}(\|u(t)\|_{s}+\|\dt u(t)\|_{s-1})\leq C(\|u_0\|_{s}+\|v_0\|_{s-1}).\end{equation}
		\end{theorem}
\medskip
\noindent{\bf Stability of the superactions:}\\
		Theorem \ref{Theorem1} shows that equation \eqref{MK} can be transformed into an equation that is in normal form up to order seven. It also shows that nonlinear terms up to this order only produce effective interactions among the modes corresponding to the same Sobolev weight. 
		As a consequence, the quantities
		$$\sum_{\substack{j\in\Z^n:|j|=k}}|q_j|^2,\quad k\in\N,$$
		called \emph{superactions} of the resulting system, are constants of motion for the truncated system \eqref{truncated}.\\ 
	Consequently, the dynamics of the transformed equation exhibits long-time stability and no effective energy transfer occurs at these orders.\\
		As for the original equation, this implies that equation \eqref{MK} exhibits long-time stability for small initial data, and that any possible energy transfer or instability mechanism can only arise from nonlinear effects of order seven or higher:
		\begin{corollary}\label{coro2}
			For sufficiently small initial data of size $\epsilon$, the superactions remains approximately constant for times of order $\epsilon^{-6}$.
		\end{corollary}
		\begin{remark}
			In the one-dimensional setting, both Theorem \ref{Theorem1} and Corollary \ref{coro2} lead to stronger consequences: thanks to the conservation of momentum, the superactions reduce to the usual actions $\{|u_j|^2\}_{j\in\Z}$, and our results translate, respectively, into the integrability of the truncated system and into the stability of actions over time scales of order $\epsilon^{-6}$.
		\end{remark}
	
	Given the properties of the resonant terms highlighted by Theorem \ref{Theorem1}, one may conjecture that they exhibit the same behaviour at all orders of the normal form. This would provide strong evidence for the integrability of equation \eqref{MK} in the one-dimensional case. 
	This is, at present, an open question, and will be the subject of future investigation.

	\subsection{Historical background}
Kirchhoff-type equations were first introduced in 1876 to model the 
transverse oscillations of elastic structures, taking into account 
the dependence of the tension on the deformation. The earliest 
rigorous contribution to the associated Cauchy problem is the seminal 
work of Bernstein \cite{Bernstein}, who in 1940 proved local 
well-posedness in $H^2\times H^1$ and global well-posedness for 
analytic initial data in the one-dimensional case.

Following Bernstein's contribution, the study of Kirchhoff-type equations 
developed along three main directions: local well-posedness, global 
well-posedness, and long-time existence.

Concerning local well-posedness, the results of Bernstein were 
extended to higher dimensions and to the periodic setting by Dickey 
\cite{Dickey}, Medeiros and Milla Miranda 
\cite{MedeirosMillaMiranda1987}, and Arosio and Panizzi 
\cite{ArosioPanizzi}, who established that the Cauchy problem 
associated with \eqref{genk} is well-posed in $H^{3/2}\times 
H^{1/2}$, with existence time of order $(\|u_0\|_{3/2} + 
\|v_0\|_{1/2})^{-2}$. Local well-posedness under analytic regularity 
assumptions, with minimal conditions on the nonlinearity $\bf f$, 
was proved through in increasing generality by Bernstein 
\cite{Bernstein}, Poho{\v z}aev \cite{P}, Arosio and Spagnolo 
\cite{arosiospagnolo}, D'Ancona and Spagnolo \cite{DAncona1992OnAA}, 
and Hirosawa \cite{Hirosawa2002DegenerateKE}.

Concerning global well-posedness, the situation differs substantially 
depending on the domain. On $\mathbb{R}^d$, global well-posedness 
with scattering for small initial data in weighted Sobolev spaces was 
proved by Greenberg and Hu \cite{Greenberg} in dimension $d=1$ and by D'Ancona and 
Spagnolo \cite{DAncona1992OnAA} in higher dimensions. On compact 
domains, no dispersion is available and the problem is substantially 
harder: for a generic nonlinearity $\bf f$, no result of global 
well-posedness in the Sobolev class is known, and the question remains 
one of the main open problems in the field. As mentioned above, the only exception is equation \eqref{MK}, for which 
global well-posedness for small initial data in $H^s\times H^{s-1}$ 
with $s\geq 2$ follows from the conservation of $K_{\bf f}$ proved 
by Poho{\v z}aev \cite{P2}.

As for the long-time existence and qualitative dynamics, several 
authors have improved the lower bounds on the existence time and 
studied recurrent or chaotic behavior. Lower bounds on the lifespan 
of small-amplitude solutions on $\mathbb{T}^n$ are obtained by Baldi 
and Haus in \cite{BaldiHaus1}, improving the bound given by standard 
local theory, and the optimality of such bounds for general initial 
data is discussed in \cite{BaldiHaus2}. Under additional 
non-resonance assumptions on the initial data, longer lifespans are 
proved in \cite{Baldi2020LongerLF}. Almost global existence and 
stability in Sobolev and Gevrey regularity for the Kirchhoff equation 
on $\mathbb{T}^1$ are established by Liu and Xiang 
\cite{liu2025globalsolutionskirchhoffequation}, using the rational 
normal form techniques of \cite{Bernier2018RationalNF}. Periodic and 
quasi-periodic solutions for forced Kirchhoff equations are 
constructed in \cite{Baldi, Corsi2018QuasiperiodicSF, 
	Montalto2017QuasiperiodicSO, Montalto2017ARR} via Nash-Moser and KAM 
techniques. Finally, Baldi, Giuliani, Guardia and Haus 
\cite{Baldi2023EffectiveCF} construct solutions whose Sobolev norms 
exhibit small chaotic oscillations over long time scales.

		\subsection{Strategy of the proof}\label{sketchof}
		
		As noted above, equation \eqref{MK} can be written as the 
		infinite-dimensional Hamiltonian system \eqref{system1}, whose analysis 
		is the central object of this paper. \\
		The strategy consists of 
		performing a sequence of changes of variables to put 
		\eqref{system1} in normal form up to order seven, and then deriving 
		energy estimates on the transformed system.
		
		\medskip
		
		\noindent\textbf{Diagonalization and block-diagonalization.}
		We start by introducing complex coordinates $(z,\bar{z})$ that 
		diagonalize the linear part of \eqref{system1}, so that the system 
		takes the form \eqref{system2}. In these coordinates, the nonlinearity 
		has a paralinear structure and the energy estimate for the Sobolev 
		norm $\|z\|_s$ takes the form \eqref{re}. However, the 
		off-diagonal part of the system, which couples $z$ with $\bar{z}$ 
		via an unbounded operator of order one, causes a loss of half a 
		derivative in the energy estimate, preventing any direct application 
		of normal form theory. This difficulty is intrinsic to the 
		quasilinear nature of \eqref{MK}.
		
		To overcome it, following the strategy introduced by Delort 
		\cite{Delort2009LONGTIMESS, Delort1} and implemented by Baldi and 
		Haus in \cite{BaldiHaus1} for the standard Kirchhoff equation, we 
		construct a nonlinear transformation $\Phi^{(2)}$ that removes the 
		unbounded off-diagonal terms, conjugating the system to a new one 
		in which all off-diagonal terms are bounded operators of order zero. 
		A first structural advantage of equation \eqref{MK} emerges already 
		at this stage: while in \cite{BaldiHaus1} the block-diagonalized 
		system contains nonlinear terms at every homogeneity order, due to 
		the presence of a square root in the diagonalizing transformation,
		in our case the resulting system \eqref{system9.1} has an exact finite 
		expansion, containing only terms up to order five. This is a direct 
		consequence of the rational structure of the nonlinearity 
		$\mathbf{f}(y)=(1+cy)^{-2}$, and it considerably simplifies the 
		subsequent normal form analysis.
		
		\medskip
		
		\noindent\textbf{First normal form step.}
		We then perform a first step of normal form, constructing a 
		bounded, close-to-identity transformation $\Phi^{(3)}$ that removes 
		all non-resonant cubic terms from the system. The construction 
		involves small divisors of the form $||j|-|k||^{-1}$, $j,k\in\Z^n$, 
		which are uniformly bounded in dimension $n=1$ and accumulate to 
		zero in dimension $n\geq 2$, causing the additional loss of half a 
		derivative in the regularity threshold $m_0$ defined in 
		\eqref{m0}. The resonant cubic terms that survive the 
		transformation commute with the linear part and, crucially, give no 
		contribution to the Sobolev energy estimates. This mirrors the 
		analogous result in \cite{BaldiHaus1} and is consistent with the 
		expected integrable behavior of \eqref{MK}.
		
		\medskip
		
		\noindent\textbf{Second normal form step.}
		In Section \ref{secondstep}, we perform a second step of normal 
		form via a transformation $\Phi^{(4)}$, removing all non-resonant 
		quintic terms. This is where the special structure of \eqref{MK} 
		manifests most clearly, in two related ways.
		
		On the one hand, a remarkable cancellation occurs among the quintic terms generated 
		by the normal form procedure: the non-integrable quintic 
		contributions  \eqref{52} and \eqref{53} cancel exactly. This is a direct 
		consequence of the structural observation made above, namely that 
		the off-diagonal quintic term produced by the block-diagonalization 
		is exactly $-Q$ times the corresponding cubic term, and it has no 
		analogue in the standard Kirchhoff equation \cite{BaldiHaus2}, since it is ultimately a 
		consequence of the rational nonlinearity of \eqref{MK}.
		
		On the other hand, and as a direct consequence of this cancellation, the 
		resonant quintic terms do not introduce small divisors of the more 
		complex form
		\begin{equation}\label{small2}
		\frac{1}{\big||j_1|\pm|j_2|\mp|j_3|\big|}, 
		\quad j_1,\,j_2,\,j_3\in\Z^n,
		\end{equation}
		which would raise the regularity threshold. Instead, the structure 
		of the small divisors at the quintic order is identical to that of 
		the cubic order, so that the regularity threshold $m_0$ remains 
		unchanged. In particular, the resonant quintic terms commute with 
		the linear part and give no contribution to the Sobolev energy 
		estimates, in the sense that
		\begin{equation}\label{boh8}
		\Re\Big(\langle\Lambda^{2s}q,\,
		\overline{X_{\mathbf{Res},5}(q,\bar{q})}\rangle\Big)=0,
		\quad s\geq m_0.
		\end{equation}
		This should be contrasted with the analogous computation for the 
		standard Kirchhoff equation carried out in \cite{BaldiHaus2}, where 
		the quintic resonant terms give a nonzero contribution to the energy 
		estimates, signalling non-integrability.
		
		\noindent\textbf{Acknowledgements.}
		We would like to thank Marina Ghisi and Massimo Gobbino for drawing our attention towards the Kirchhoff-Poho{\v z}aev equation and for several interesting discussions.
		Our research was supported by PRIN 2020XB3EFL ``Hamiltonian and dispersive PDEs'', by PRIN 2022HSSYPN ``Turbulent effects vs Stability in Equations from Oceanography'', and by INdAM-GNAMPA.

		\section{Functional setting and Hamiltonian formalism}\label{functional setting}
		
		In this section we introduce the Hamiltonian formulation of the Kirchhoff-Poho{\v z}aev equation.
		Although Kirchhoff-type equations are often studied from a purely PDE 
		perspective, the associated dynamics can be naturally embedded into an 
		infinite-dimensional Hamiltonian framework.
		 This viewpoint is fundamental 
		for investigating structural properties of the equation (such as the 
		presence of conserved quantities and an underlying symplectic structure) and provides the natural setting for the application of normal form 
		and dynamical systems techniques.
		
		\subsection{Sobolev spaces and Fourier expansions}
		
		We expand a function $u\colon\T^n\to\R$ in Fourier series
		\begin{equation*}
			u(x)=\frac{1}{(2\pi)^{n/2}}\sum_{k\in\Z^n}u_ke^{ik\cdot x},
			\qquad
			u_k=\frac{1}{(2\pi)^{n/2}}\int_{\T^n}u(x)e^{-ik\cdot x}\,dx,
		\end{equation*}
		and we identify $u$ with the sequence $\{u_k\}_{k\in\Z^n}$ of its 
		Fourier coefficients.\\ We work in the Sobolev space
		\[
		H^s(\T^n,\C)
		:=\Bigl\{u\colon\T^n\to\C
		\;:\;
		\|u\|_s<\infty
		\Bigr\},
		\]
		and in its subspace of 
		real-valued functions
		\[
		H^s(\T^n,\R)
		:=\bigl\{u\in H^s(\T^n,\C)\;:\;\bar u_j=u_{-j}\bigr\},
		\]
		where 
	\begin{equation}\label{norma1}
		\|u\|_s:=\left(\sum_{j\in\Z^n}\langle j\rangle^{2s}|u_j|^2\right)^{1/2}	,\quad\langle j\rangle:=\max\{1,|j|\},\,\,|j|=|j_1|+\ldots+|j_n|.
	\end{equation}
	Moreover, in terms of the Fourier coefficients, equation~\eqref{MK} becomes the 
		infinite-dimensional system of ODEs
		\begin{equation}\label{fourier2in}
			\ddot{u}_k(t)
			+\frac{|k|^2}{1+c\sum_{j\in\Z^n}|j|^2|u_j(t)|^2}\,u_k(t)=0,
			\qquad k\in\Z^n.
		\end{equation}
		The equation for the index $k=0$ reduces to $\ddot{u}_0(t)=0$, so the 
		zero mode $u_0$, which encodes the spatial mean 
		$\int_{\T^n}u(t,x)\,dx$, decouples from the rest of the system. It is 
		therefore not restrictive to work in the space of zero-mean functions.\\
		
		Accordingly, we introduce the space of complex zero-mean Sobolev 
		functions
		\begin{equation*}
			H_0^s(\T^n,\C)
			:=\left\{
			u(x)=\frac{1}{(2\pi)^{n/2}}
			\sum_{k\in\Z^n\setminus\{0\}}u_k e^{ik\cdot x}
			\;:\;u_k\in\C,\;\|u\|_s<\infty
			\right\},
		\end{equation*}
		together with its real subspace $H_0^s(\T^n,\R)$.\\
		Let us note that, in this setting, the norm \eqref{norma1} reads
		\begin{equation}\label{norma2}
           \|u\|_s:=\left(\sum_{j\in\Z^n}|j|^{2s}|u_j|^2\right)^{1/2}.
		\end{equation}\\
		Moreover, we denote by
		\[
		H^s_0(\T^n,\mathrm{c.c.}):=\bigl\{(f,g)\in H^s_0(\T^n,\C)\times H^s_0(\T^n,\C)
		:\ g=\bar f\,\bigr\}
		\]
		the space of pairs of complex conjugate, zero-mean functions, endowed with the
		norm $\|(f,\bar f)\|_s:=\|f\|_s$.\\
		Finally, for $s\in\R$ and $r>0$, we write $B_r(H^s(\T^n,\C))$ for the open 
		ball of radius $r$ centered at the origin of $H^s(\T^n,\C)$. \\
		Analogous notation is used for $H^s(\T^n,\R)$, $H_0^s(\T^n,\C)$, $H_0^s(\T^n,\R)$ and $	H^s_0(\T^n,\mathrm{c.c.})$.
		
		\subsection{Hamiltonian formulation}
		
	By adopting the set of variables $(u,\,v)=(u,\,u_t)$ and considering the Hamiltonian
	\begin{equation}\label{hamiltonian1}
		H_\R(u,v)=\frac{1}{2}\int_{\T^n}v^2dx-\frac{1}{2c}\left(\dfrac{1}{1+c\int_{\T^n}\left|\nabla u\right|^2dx}\right),
	\end{equation}
    $H_0^1(\T^n,\R)\times L^2(\T^n,\R)\rightarrow \R$, we can write \eqref{system1} as the (infinite-dimensional) Hamiltonian system
	\begin{equation*}\label{system11}
	\left(\begin{matrix}
		\partial_{t}u \\ \partial_{t}v
	\end{matrix}\right)
	=J\nabla {H_{\mathbb{R}}}(u,v)=
	J\left(\begin{matrix}
		\partial_{u}H_{\mathbb{R}}(u,v)\\
		\partial_{v}H_{\mathbb{R}}(u,v)
	\end{matrix}\right)\,,
	\quad 
	J=\begin{pmatrix}
		{0}&{1}\\
		{-1}&{0}
		\end{pmatrix}.
	\end{equation*}
	Here, $\partial_uH_\R$ and $\partial_vH_\R$ are the gradients of $H_\R$ with respect to the real scalar product on $L^2$:
	\begin{equation*}\label{jap}
		\jap{f,g} :=\int_{\T^n}f(x)g(x)dx,\quad f,\,g\in L^2(\T^n,\R).
	\end{equation*}
	From a dynamical systems point of view, this implies that equation \eqref{MK} can be seen as an infinite-dimensional Hamiltonian system in the phase space $H_0^s(\T^n,\R)\times H_0^{s-1}(\T^n,\R)$, $s\geq1$.\\
	
	Indeed we have 
	\begin{equation}\label{eq:1.14bis}
		\mathrm{d}H_{\mathbb{R}}(u,v)\cdot W
		=
		\Omega_{\mathbb{R}}(X_{H_{\mathbb{R}}}(u,v),W)
	\end{equation}
	for any 
	$W \in H^{1}(\mathbb{T}^{n};\mathbb{R})\times L^{2}(\mathbb{T}^{n};\mathbb{R})$, 
	where $\Omega_{\mathbb{R}}$ 
	is the non-degenerate symplectic form
	\[
	\Omega_{\mathbb{R}}(W_1,W_2) 
	= \int_{\mathbb{T}}(u_1v_2-v_1u_2)dx\,,
	\quad \forall \;W_1=\vect{u_1}{v_1}\,,\;W_2=\vect{u_2}{v_2}\in H^{1}(\mathbb{T}^{n};\mathbb{R})\times L^{2}(\mathbb{T}^{n};\mathbb{R}).
	\]
	The Poisson brackets between two Hamiltonian 
	$H_{\mathbb{R}}, G_{\mathbb{R}}: 
	H^{1}(\mathbb{T}^{n};\mathbb{R})\times L^{2}(\mathbb{T}^{n};\mathbb{R})\to \mathbb{R}$
	are defined in the classical manner as
	\begin{equation}\label{realpoisson}
		\{H_{\mathbb{R}},G_{\mathbb{R}}\}
		:=\Omega_{\mathbb{R}}(X_{H_{\mathbb{R}}},X_{G_{\mathbb{R}}})\,.
	\end{equation}
	Let us note, moreover, that the nonlinearity in Hamiltonian \eqref{hamiltonian1} and the corresponding one in equation \eqref{MK}, are well-defined only when the denominator
	\begin{equation*}
		1+c\int_{\T^n}\left|\nabla u\right|^2dx\neq0.
	\end{equation*}
	Since $\int_{\T^n}\left|\nabla u\right|^2=\|u\|_1^2$, it is sufficient to consider solutions supported in the ball
	\begin{equation*}\label{cond2}
		u\in B_{\dc}\left(H_0^{1}(\T^n,\R)\right),
	\end{equation*}
	where 
	\begin{equation}\label{deps}
		\dc:=\frac{1}{2\sqrt{|c|}}.
	\end{equation}
	In this way, the nonlinearities involved in Hamiltonian \eqref{hamiltonian1} and in equation \eqref{MK} are well-defined.
	\subsection{Symmetries of Kirchhoff-type equations}\label{Mjs}
	We recall a couple of
	structural properties of Kirchhoff-type equations. Although not central to the aims of this paper, these properties are of
	independent interest and
	turn out to
	be particularly useful when performing a Birkhoff normal form algorithm.
	\begin{itemize}
		\item[{\bf (1)}]{\bf First integrals of Kirchhoff-type equations:}\\
		If we consider the Hamiltonian momentum
		\begin{equation}\label{Mi}
			M:=\int_{\T^n}(\dt u)\nabla udx,
		\end{equation}
		a simple calculation shows that it is conserved by the dynamics generated by the Hamiltonian $H$.\\
		In Fourier coordinates we have that
		\begin{equation*}
			M=i\sum_{j\in\Z}ju_j\big(\dt u_{-j}\big)=\dfrac{i}{2}\sum_{j\in\Z}j\big[u_j(\dt u_{-j})-u_{-j}(\dt u_j)\big].
		\end{equation*}
		The momentum is hence the sum $M=\sum\limits_{j\in\Z}M_j$, of infinitely many quantities 
		\begin{equation*}\label{Mj}
			M_j:=\dfrac{i}{2}j\big[u_j(\dt u_{-j})-u_{-j}(\dt u_j)\big],
		\end{equation*}
		moreover, thanks to the special structure of Kirchhoff-type equations, all the $M_j$ are first integrals. In fact, one has
		\begin{equation*}
			\dt M_j = \dfrac{i}{2}j\big[u_j(\dtt u_{-j})-u_{-j}(\dtt u_j)\big]
			\stackrel{\eqref{fourier2in}}{=}0.
		\end{equation*}
		\item[{\bf (2)}]{\bf Invariance of the Fourier support:}\\
		If $(u,u_t)$ is a solution of \eqref{MK}, the  Fourier support at a given time $t$ is defined as the set of indices
		\begin{equation}\label{fouriersupport}
			\mathcal{S}(t) = \{ j \in \mathbb{Z}^n: (u_j(t), \dt u_j(t)) \neq (0, 0) \}.
		\end{equation}
		\begin{lemma}\label{invR}
			Let $(u,u_t)$ be a solution of \eqref{MK} on a time interval $[0,T]$, with initial data $(u_0,v_0)$. Then
			$$\mathcal{S}(t)=\mathcal{S}(0)\quad\text{ for al }t.$$
		\end{lemma}
		\begin{proof}
		From \eqref{fourier2in} we have that if $u_{0,j}=v_{0,j}=0$ for some $j$, then it must hold $u_j(t)=0$ for all $t$.
		\end{proof}
		This in particular implies that initial data supported on a finite subset of modes, give rise to global-in-time solutions, since in this case, equation \eqref{MK} reduces to a finite-dimensional Hamiltonian system with an analytic Hamiltonian.
	\end{itemize}

	\section{Preliminary transformations}\label{Preliminary Transformations}
	The purpose of this section is to perform two preliminary changes of variables, one linear and one nonlinear, in order to conjugate system \eqref{system1} to one in a more convenient form.\\
	The first one diagonalizes the linear part of the system, while the second one has the removes, up to a bounded remainder, the off-diagonal unbounded part.\\
	This strategy of eliminating the off-diagonal unbounded terms before the normal form construction is necessary because of the quasi-linear structure of the Kirchhoff equation.
	
	\subsection{Diagonalization of the highest order}
	We want to diagonalize the linear part of system \eqref{system1}, namely
	\begin{equation}\label{system111}
		\begin{cases}
			\dt u=v\\
			\dt v=\Delta u
		\end{cases}.
	\end{equation}
	Let us consider the complex linear isomorphism
		\begin{equation}\label{change}
		\begin{aligned}
	\Phi^{(1)}&:\, H_0^{s}(\T^n,\,\C)\times H_0^{s}(\T^n,\,\C)\longrightarrow  H_0^{s+\frac12}(\T^n,\,\C)\times H_0^{s-\frac12}(\T^n,\,\C),\\
		\Phi^{(1)}&(z_1,z_2)=\left(\dfrac{\Lambda^{-\frac{1}{2}}\left(z_1+z_2\right)}{\sqrt{2}},\,\dfrac{\Lambda^{\frac{1}{2}}\left(z_1-z_2\right)}{i\sqrt{2}}\right)
		\end{aligned}
	\end{equation}
	where $\Lambda=\sqrt{-\Delta}$  is the Fourier multiplier defined by
	$$\Lambda e^{ik\cdot x}=|k|e^{ik\cdot x},\quad\text{ for }k\in\Z^n.$$
Note that, when $\Phi^{(1)}$ is restricted to the space of complex-conjugate, zero-mean functions $H_0^s(\T^n,c.c.)$, it becomes a real isomorphism into $H^{s+\frac12}_0(\T^n,\R)\times H^{s-\frac12}_0(\T^n,\R)$, moreover, the following estimate holds
\begin{equation}\label{bond1}
\left\|\Phi^{(1)}(z,\z)\right\|_{H_0^{s+\frac12}\times H_0^{s-\frac12}}\leq\|(z,\z)\|_s.
\end{equation}\\\\
By defining the new set of variables
	$$(u,v)=\Phi^{(1)}(z,\z),$$ the linear part \eqref{system111} takes the diagonal form
	$$
	\begin{cases}
		\dt z=-i\Lambda z\\
		\dt \z=i\Lambda \z
	\end{cases},
	$$
	while the whole system \eqref{system1} transforms into
	\begin{equation}\label{system2}
		\begin{system}
			\dt z=-i\Lambda z-\dfrac{i}{2}\left(\dfrac{1}{\left(1+\frac{c}{2}\langle\Lambda(z+\z),z+\z\rangle\right)^2}-1\right)\left(\Lambda z+\Lambda \z\right) \\
			\dt \z=i\Lambda \z+\dfrac{i}{2}\left(\dfrac{1}{\left(1+\frac{c}{2}\langle\Lambda(z+\z),z+\z\rangle\right)^2}-1\right)\left(\Lambda z+\Lambda \z\right) 
		\end{system}\,,
	\end{equation}
	where 
	$$\langle\Lambda(z+\z),z+\z\rangle=\int_{\T^n}\left[(z(x)+\z(x))\left(\Lambda z(x)+\Lambda\z(x)\right)\right]dx.$$
	Finally, by defining
	\begin{equation}\label{mu}
	\mu(x)=\frac{1}{\left(1+cx\right)^2}-1,\quad Q(z,\z)=\frac{1}{2}\left\langle\Lambda(z+\z),z+\z\right\rangle,
	\end{equation} system \eqref{system2} can be rewritten as
	\begin{equation}\label{system3}
		\begin{cases}
			\dt z=-i\Lambda z-\frac{i}{2}\mu\left(Q(z,\z)\right)\left(\Lambda z+\Lambda \z\right) \\
			\dt \z=i\Lambda \z+\frac{i}{2}\mu\left(Q(z,\z)\right)\left(\Lambda z+\Lambda \z\right) 
		\end{cases}\,.
	\end{equation}
	Note that the second equation of \eqref{system3} is redundant, being the complex conjugate of the first one.\\
	\begin{remark}\label{epsc}
		The quantity $Q(z,\z)$ is real and positive, since it corresponds, in the new variables, to the term $\int_{\T^n}|\nabla u|^2dx$.\\ 
		Therefore, system \eqref{system3} is well-defined only when 
		\begin{equation}\label{cond1}
			Q(z,\z)<1/|c|.
		\end{equation}
		Since one has
		\begin{equation}\label{Q12}
			Q(z,\z)\leq\|z\|^2_{\frac{1}{2}},\quad (z,\z)\in H_0^{\frac{1}{2}}(\T^n,c.c.)
		\end{equation}
		condition \eqref{cond1} is satisfied for $(z,\z)$ in the ball
		\begin{equation*}
			(z,\z)\in B_{\dc}\left(H_0^{\frac{1}{2}}(\T^n,c.c)\right),
		\end{equation*}
		where 
		$$B_r\left(H_0^{\frac{1}{2}}(\T^n,c.c)\right):=\{(z,\z)\,:\,\|z\|_{\frac{1}{2}}\leq r\}$$
		and $\dc$ is the same as in \eqref{deps}.
	\end{remark}
	
	System \eqref{system3} is again Hamiltonian of the form
	\begin{equation}\label{System4}
		\begin{cases}
			\dt z=-i\nabla_{\z} H(z,\z)\\\\
			\dt \z=i\nabla_z H(z,\z)
		\end{cases}\,
	\end{equation}
	with
	\begin{eqnarray*}\label{Ham1}
		&H:&\,H^{\frac{1}{2}}(\T^n,\C)\times H^{\frac{1}{2}}(\T^n,\C)\longrightarrow \R\nonumber\\\\
		&H(z,\z)=&\langle\Lambda z,\z\rangle-\frac{1}{2c}\left[\dfrac{\left(\frac{c}{2}\langle\Lambda(z+\z),z+\z\rangle\right)^2}{1+\frac{c}{2}\langle\Lambda(z+\z),z+\z\rangle}\right].\nonumber
	\end{eqnarray*}
	
	Here, $\nabla_z$ and $\nabla_{\bar z}$ denote the gradients with respect to the complex inner product
	\begin{equation*}\label{cprod}
		\jap{f,\,g}_\C=\jap{f,\,\bar{g}},\quad f,\,g\in L^2(\T^n,\C).
	\end{equation*}
	The map $\Phi^{(1)}$ is also a symplectomorphism, in the sense that the conjugated Hamiltonian system \eqref{System4} preserves the same structure as the previous one indeed it can be written in the form
	$$\dt\begin{pmatrix}
		z\\
		\z
	\end{pmatrix}=iJ
	\begin{pmatrix}
		\nabla_z H\\
		\nabla_{\z} H
	\end{pmatrix},\quad\quad
	J=\begin{pmatrix}
	0&I\\
	-I&0
	\end{pmatrix}.
	$$
	Moreover, the Poisson bracket \eqref{realpoisson} becomes
	\begin{equation*}
		\begin{aligned}
			\{H,G\} 
			={\rm i} \int_{\mathbb{T}^{d}}
			\big(\partial_{z}G\partial_{\bar{z}}H-
			\partial_{\bar{z}}G\partial_{z}H\big) \mathrm{d}x\,.
		\end{aligned}
	\end{equation*}
	Finally, this set of coordinates turns out to be particularly useful for studying the structure of the resonant terms. In fact, since system \eqref{System4} has the form
	\begin{equation}\label{rs}
	\dt
	\begin{pmatrix}
	z\\
	\z
	\end{pmatrix}=\F(z,\z)=
	\begin{pmatrix}
	\F_1(z,\z)\\
	\F_2(z,\z)
	\end{pmatrix},
	\end{equation}
	where the vector field $\F$ has the real structure $\F_2(z,\z)=\overline{\F_1(z,\z)}$, one has that the energy estimates for its solutions yield:
	\begin{equation}\label{re}
	\dt\|z\|_s^2=\dt\jap{\Lambda^sz,\Lambda^{s}\z}=2\Re\Big(\jap{\Lambda^{2s}z,\overline{\F_1(z,\z)}}\Big).
	\end{equation}
	\begin{remark}{\bf (Conserved quantities in complex variables)}\label{MjC}\\
		From Subsection \ref{Mjs} we know that the quantities
		$$M_j:=\dfrac{1}{2}ij\big[u_j(\dt u_{-j})-u_{-j}(\dt u_j)\big],\quad j\in\Z^n$$
		are conserved for the flow of \eqref{MK}, where $\{u_j\}_{j\in\Z}$ are the Fourier coefficients for the solution of \eqref{MK} in the original variables.\\
		In the complex coordinates given by \eqref{change}, the quantities $M_j$ take the form
		\begin{eqnarray*}\label{Mc}
			M_j&=&\dfrac{1}{4}j\left[(z_j+\z_{-j})(z_{-j}-\z_j)-(z_{-j}+\z_j)(z_j-\z_{-j})\right]\nonumber\\
			&=&\dfrac{1}{2}j[z_{-j}\z_{-j}-z_j\z_j]\\
			&=&\dfrac{1}{2}j[|z_{-j}|^2-|z_j|^2]\nonumber.
		\end{eqnarray*}
		We then have that the quantities $|z_j|^2-|z_{-j}|^2$, $j\in\N^n$ are conserved by the flow of \eqref{MK}.
	\end{remark}
	\vspace{0.5cm}
	\begin{remark}{\bf (Invariance of the Fourier support in complex variables)}\label{invC}\\
		The Fourier support for the original variables, defined in \eqref{fouriersupport}, translates, in the new complex variables, into the set
		\begin{equation*}\label{fouriersupportc}
			\S_\C(t):=\{j \in \mathbb{Z}^n: (z_{j}(t),\z_{-j}(t)) \neq (0, 0) \}.
		\end{equation*}
		This, together with Lemma \ref{invR}, implies that if the initial data $z_0$ for system \eqref{system2} is supported on a symmetric set of frequencies, then the support remains invariant for the corresponding solution.
	\end{remark}
	\vspace{0.5cm}
	Even though the linear operator is now diagonal, an energy estimate on the Sobolev norm $H^s$, $s\geq\frac{1}{2}$ of a solution $z$ of \eqref{system3} yields
	\begin{eqnarray*}
		\dt\|z\|_s^2&\stackrel{\eqref{re}}{=}&2\Re\left(i\jap{\Lambda^{2s} z,\,\Lambda \z+\frac{1}{2}\mu\left(Q(z,\z)\right)\left(\Lambda z+\Lambda \z\right) }\right)\\
		&=&\left|\mu\left(Q(z,\z)\right)\right|\Re\left(i\jap{\Lambda^{2s} z,\,\Lambda z}\right).
	\end{eqnarray*}
	Since
	\begin{equation*}\label{stimQ}
		\left|\mu\left(Q(z_1,\z_2)\right)\right|\leq \c_0\|z_1\|_{\frac{1}{2}}\|z_2\|_{\frac{1}{2}}
	\end{equation*}
	for $ z_1,\,z_2\in B_{\dc}\left(H^{\frac{1}{2}}(\T^n,c.c.)\right)$ and $\c_0>0$ depending on $c$, one has an additional loss of half derivative
	$$\dt\|z\|_s^2\leq\c_0(\epsilon,c)\|z\|^2_{\frac{1}{2}}\|z\|_{s+\frac{1}{2}}^2\,,$$
	which  prevents one from performing an energy estimate.\\
	The terms responsible for this loss are the off-diagonal terms of system \eqref{system3}, i.e. those of the form
	$$\begin{cases}
		\dt z=\ldots-\frac{i}{2}\mu\left(Q(z,\z)\right)\Lambda \z \\
		\dt \z=\ldots+\frac{i}{2}\mu\left(Q(z,\z)\right)\Lambda z
	\end{cases}.$$
	We then have to perform a preliminary step in order to remove those terms.
	\subsection{Block diagonalization}\label{block diagonalization}
	The aim of this subsection is to transform system \eqref{system3} into the system \eqref{system9.1} below, which is block-diagonal modulo bounded remainders.\\
	The main result of this subsection is the following:
	\begin{theorem}\label{lemmablock}
		Let $\delta_0$ be as in \eqref{deps}. There exists an invertible, close-to-identity map 
		$$\Phi^{(2)}:B_{\delta_0}\left(H_0^s(\T^n,c.c.)\right)\rightarrow H_0^s(\T^n,c.c)$$
		that conjugates system \eqref{system3} to one of the form
		\begin{equation}\label{system9.1}
			\begin{cases}
				\dt\eta=-i\left(1-cQ(\eta,\eb)\right)\Lambda\eta-\dfrac{ic\left(1-cQ(\eta,\eb)\right)}{2}\left(\left\langle\Lambda\eb,\Lambda\eb\right\rangle-\left\langle\Lambda\eta,\Lambda\eta\right\rangle\right)\eb\\\\
				\dt\eb=i\left(1-cQ(\eta,\eb)\right)\Lambda\eb-\dfrac{ic\left(1-cQ(\eta,\eb)\right)}{2}\left(\left\langle\Lambda\eb,\Lambda\eb\right\rangle-\left\langle\Lambda\eta,\Lambda\eta\right\rangle\right)\eta
			\end{cases}.
		\end{equation}
	\end{theorem}
\begin{remark}
A first structural advantage of equation \eqref{MK} already emerges at this
stage: unlike the block-diagonalized system in \cite{BaldiHaus1}, which contains
nonlinear terms at every order of homogeneity, in our case the resulting system
\eqref{system9.1} has an exact finite expansion, containing only terms up to
order five (compare the factor $(1-cQ)$ appearing here with the factors $\sqrt{1+2P}$ and $(1+2P)^{-1}$ appearing in equation (3.13) of \cite{BaldiHaus1}). This reflects the rational structure
of the nonlinearity $f(y)=(1+cy)^{-2}$ and considerably simplifies the
subsequent normal form analysis.
\end{remark}
We first rewrite system \eqref{system3} in the equivalent matrix form
\begin{equation}\label{matrix1}
\dt
\begin{pmatrix} z\\ \bar z\end{pmatrix}
=i\,\mathcal{L}(z,\bar z)
\begin{pmatrix} \Lambda z\\ \Lambda \bar z\end{pmatrix},
\qquad
\mathcal{L}(z,\bar z):=-\frac12
\begin{pmatrix}
2+\nu(z,\bar z) & \nu(z,\bar z)\\
-\nu(z,\bar z) & -2-\nu(z,\bar z)
\end{pmatrix},
\end{equation}
where
\begin{equation}\label{Q}
\nu(z,\bar z):=\dfrac{1}{2}\mu\bigl(Q(z,\bar z)\bigr)
\end{equation}
and $\mu,Q$ are defined in \eqref{mu}.\\
	Since the off-diagonal terms of $\L$ are responsible for the loss of half a derivative in performing an energy estimate on the $H_0^s$ norm of the solution of \eqref{matrix1}, our aim is to construct a transformation that removes such terms.\\
	By adopting a paradifferential point of view, one can consider system \eqref{matrix1} as a linear system for which $\nu(z,\z)$ plays the role of a coefficient. Our strategy therefore is to diagonalize the associated matrix 
	\begin{equation*}\label{matrix2}
		\Mat=\begin{pmatrix}
			-1-\y & -\y\\
			\y & 1+\y
		\end{pmatrix}
	\end{equation*}
	where $\y=\nu(z,\z)$.\\
	The matrix $\Mat$ has eigenvalues of the form
	$$\texttt{m}_1=-\sqrt{1+2\y} \quad\text{and}\quad\texttt{m}_2=\sqrt{1+2\y},$$
	with associated eigenvectors
	$\begin{pmatrix}
		1\\
		\rho(\y)
	\end{pmatrix}$, $ \begin{pmatrix}
		\rho(\y)\\
		1
	\end{pmatrix}$
	where
	\begin{equation}\label{rho}
		\rho(\y)=\frac{-\y}{1+\y+\sqrt{1+2\y}}.
	\end{equation}
	Therefore we have
	\begin{equation}\label{diagL}
		\begin{pmatrix}
			1 & \rho(\y)\\
			\rho(\y) & 1
		\end{pmatrix}^{-1}
		\begin{pmatrix}
			-1-\y & -\y\\
			\y & 1+\y
		\end{pmatrix}
		\begin{pmatrix}
			1 & \rho(\y)\\
			\rho(\y) & 1
		\end{pmatrix}
		=
		\begin{pmatrix}
			-\sqrt{1+2\y} & 0\\
			0 & \sqrt{1+2\y}
		\end{pmatrix}.
	\end{equation}

	It is natural to define a new set of variables of the form
	\begin{eqnarray}\label{change3.1}
		\begin{pmatrix}
			z_1\\
			z_2
		\end{pmatrix}&=&\M(z_1,z_2)
		\begin{pmatrix}
			\eta_1\\
			\eta_2
		\end{pmatrix},\nonumber
		\\ \\\M(z_1,z_2)&:=&\frac{1}{\sqrt{1-\rho^2\left(\nu(z_1,z_2)\right)}}
		\begin{pmatrix}
			1&\rho\left(\nu(z_1,z_2)\right)\\
			\rho\left(\nu(z_1,z_2)\right)&1
		\end{pmatrix}.\nonumber
	\end{eqnarray}
	\begin{remark}
		As pointed out in \cite{BaldiHaus1}, the corrective term 
		$$\frac{1}{\sqrt{1-\rho^2\left(\nu(z_1,z_2)\right)}}$$
		used in formula \eqref{change3.1} is essential. Indeed, using any other similar change would generate a diagonal term of order zero in the new system, which gives a non-trivial contribution to the energy estimate.\\
		This corrective term is the only way, up to constant factors, to eliminate those diagonal terms and is related to the symplectic structure of \eqref{system9.1}.
	\end{remark}
	\begin{remark}\label{real structure}
		We can invert expression \eqref{change3.1} with respect to the variables $z_1,z_2$ and get
		\begin{equation}\label{realstr}\begin{pmatrix}
				\eta_1\\
				\eta_2
			\end{pmatrix}=\M^{-1}(z,\z)\begin{pmatrix}
				z_1\\
				z_2
			\end{pmatrix},
		\end{equation}
		where
		\begin{equation}\label{Mm1}
			\M^{-1}(z_1,z_2)=\dfrac{1}{\sqrt{1-\rho^2(\nu(z_1,z_2))}}
			\begin{pmatrix}
				1&-\rho(\nu(z_1,z_2))\\
				-\rho(\nu(z_1,z_2))&1
			\end{pmatrix}.
		\end{equation}
		From \eqref{realstr} it follows that complex-conjugated pairs of functions $(z,\z)$ are mapped into another pair of complex-conjugated functions. Therefore, \eqref{change3.1} is well-defied as a map from the space of complex-conjugate functions onto itself. We will then consider $\eta_2=\bar{\eta}_1$.
	\end{remark}
	\begin{remark}
		From the expression of $\rho$ in \eqref{rho} and $\q$ in \eqref{Q} one has that
		\begin{equation}\label{rhoq}
			\rho(\nu(z,\z))=\dfrac{cQ(z,\z)}{2+cQ(z,\z)}.
		\end{equation}
		It follows that the above quantity is well-defined, since, from Remark \ref{epsc}, $(z,\z)\in B_{\dc}\left(H^{\frac12}_0(\T^n,c.c.)\right)$.\\
		Let us note moreover that, under this hypothesis, also the term $\frac{1}{\sqrt{1-\rho^2\left(\nu(z,\z)\right)}}$ is well-defined.
	\end{remark}
	Since the change defined in \eqref{change3.1} is implicit in the variable $(z,\z),$ we have to express $\nu(z,\z)$ in terms of $\eta$ and $\eb$:\\
	since $\nu(z,\z)=\frac{1}{2}\mu(Q(z,\z))$, we start by making explicit the dependence of $Q(z,\z)=\frac{1}{2}\jap{\Lambda(z+\z),\,z+\z}$ on the variables $(\eta,\eb)$.\\
	Recalling \eqref{change3.1}, we have
	\begin{eqnarray}\label{rho2}
		\left\langle\Lambda(z+\z),z+\z\right\rangle&=&\frac{1}{(1-\rho^2(\nu(z,\z)))}\left\langle(1+\rho(\nu(z,\z)))\Lambda(\eta+\eb),(1+\rho(\nu(z,\z)))(\eta+\eb)\right\rangle\nonumber\\\\
		&=&\frac{1+\rho(\nu(z,\z))}{1-\rho(\nu(z,\z))}\left\langle\Lambda(\eta+\eb),\eta+\eb\right\rangle,\nonumber
	\end{eqnarray}
	moreover, from the definition of $\rho$ in \eqref{rho}, we have that
	$$\frac{1-\rho(\y)}{1+\rho(\y)}=\sqrt{1+2\y},$$
	hence
	\begin{equation}\label{impl1}
		Q(\eta,\eb)\stackrel{\eqref{rho2}}{=}\sqrt{1+\nu(z,\z)}Q(z,\z)
		\stackrel{\eqref{mu}}{=}\frac{Q(z,\z)}{1+cQ(z,\z)}.
	\end{equation}
	By inverting expression \eqref{impl1} with respect to $Q(z,\z)$ we get 
	\begin{equation}\label{phi}
		Q(z,\z)=\frac{Q(\eta,\eb)}{1-cQ(\eta,\eb)}:=\phi\left(Q(\eta,\eb)\right).
	\end{equation}
	Let us note that \eqref{phi} is well-defined only when $Q(\eta,\eb)<1/|c|$. As in Remark \ref{epsc} it is sufficient to impose
	\begin{equation*}\label{epsc1}
		(\eta,\eb )\in B_{\dc}\left(H_0^{\frac{1}{2}}(\T^n,c.c.)\right).
	\end{equation*}
	We can now make expression \eqref{change3.1} explicit with respect to the variables $(\eta,\,\eb)$ and define the map
	\begin{equation}\label{Phi2}
		(z,\z)=\Phi^{(2)}(\eta,\eb):=\M(\eta,\eb)
		\begin{pmatrix}
			\eta\\
			\eb
		\end{pmatrix},
	\end{equation}
	where
	$$
	\M(\eta,\eb)=\frac{1}{\sqrt{1-\rh^2\left(\eta,\eb\right)}}
	\begin{pmatrix}
		1&\rh\left(\eta,\eb\right)\\
		\rh\left(\eta,\eb\right)&1
	\end{pmatrix}
	$$
	and
	\begin{eqnarray}\label{rho22}
		\rh(\eta,\eb)&:=&\rho\left(\frac{\mu}{2}\left(\phi\left(Q(\eta,\eb)\right)\right)\right)\nonumber\\\\
		&\stackrel{\eqref{mu},\eqref{phi}}{=}&\frac{cQ(\eta,\eb)}{2-cQ(\eta,\eb)}.\nonumber
	\end{eqnarray}

	Reasoning as in \cite[Lemma 3.1]{BaldiHaus1}, we have:
	\begin{lemma}\label{lemma1}
		For every $s\geq\frac{1}{2}$ the nonlinear map $\Phi^{(2)}:B_{\dc}\left(H^s_0(\T^n,\,c.c.)\right)\rightarrow H^s_0(\T^n,\,c.c.)$ is invertible, continuous, with continuous inverse
		$$\left(\Phi^{(2)}\right)^{-1}(z,\z)=\dfrac{1}{\sqrt{1-\rho^2\left(\nu(z,\z)\right)}}
		\begin{pmatrix}
			1&-\rho\left(\nu(z,\z)\right)\\
			-\rho\left(\nu(z,\z)\right)&1
		\end{pmatrix}
		\begin{pmatrix}
			z\\
			\z
		\end{pmatrix}
		$$
		defined for every $(z,\z)\in B_{\dc}\left(H^s_0(\T^n,\,c.c.)\right)$ with $\dc$ defined in \eqref{deps}.\\
		Moreover, for all $s\geq\frac{1}{2}$, all $(\eta,\eb)\in H^s_0(\T^n,\,c.c.)$ satisfying $\|\eb\|_{\frac{1}{2}}\leq\dc$  one has
		\begin{equation}\label{bond2}\|\Phi^{(2)}(\eta,\eb)\|_s\leq \mathfrak{C}(\|(\eta,\eb)\|_{\frac{1}{2}})\|(\eta,\eb)\|_s
		\end{equation}
		for some increasing function $\mathfrak{C}$. The same estimate is satisfied by $\left(\Phi^{(2)}\right)^{-1}$.
	\end{lemma}
	\begin{proof}
		The bounds $\|\eta\|_{\frac{1}{2}},\,\|\eta\|_{\frac{1}{2}}\leq\dc$ guarantee that the denominators $\sqrt{1-\rho^2\left(\nu(z,\z)\right)}$, $\sqrt{1-\rho^2\left(\nu(\eta,\eb)\right)}$ are well defined, while the estimates on $\Phi^{(2)}$ and $\left(\Phi^{(2)}\right)^{-1}$ follow from \eqref{rho}, \eqref{rhoq}, \eqref{Phi2} and estimate \eqref{Q12}.
	\end{proof}
	Let us see how system \eqref{matrix1} behaves under the change \eqref{change3.1}.
	By \eqref{Phi2} one has
	\begin{eqnarray}\label{sysS}
		\dt\begin{pmatrix}
			z\\
			\z
		\end{pmatrix}&=&\dt\Phi^{(2)}(\eta,\eb)=
		\dt\left(\M(\eta,\eb)\begin{pmatrix}
			\eta\\
			\eb
		\end{pmatrix}\right)\nonumber\\
		&=&
		\M(\eta,\eb)\dt
		\begin{pmatrix}
			\eta\\
			\eb
		\end{pmatrix}
		+
		\dt\Big(\M(\eta,\eb)\Big)
		\begin{pmatrix}
			\eta\\
			\eb
		\end{pmatrix}.
	\end{eqnarray}
	On the other hand, by using \eqref{matrix1},
	\begin{eqnarray}\label{sysD}
		\dt\begin{pmatrix}
			z\\
			\z
		\end{pmatrix}
		&=&i\L(z,\z)
		\begin{pmatrix}
			\Lambda z\\
			\Lambda\z
		\end{pmatrix}\nonumber\\\nonumber\\
		&=&i
		\begin{pmatrix}
			-1-\nu(z,\z) & -\nu(z,\z)\\
			\nu(z,\z) & 1+\nu(z,\z)
		\end{pmatrix}
		\begin{pmatrix}
			\Lambda z\\
			\Lambda \z
		\end{pmatrix}\nonumber\\\nonumber\\
		&\stackrel{\eqref{change3.1}}{=}&
		i
		\begin{pmatrix}
			-1-\nu(z,\z) & -\nu(z,\z)\\
			\nu(z,\z) & 1+\nu(z,\z)
		\end{pmatrix}\M(z,\z)
		\begin{pmatrix}
			\Lambda\eta\\
			\Lambda\eb
		\end{pmatrix}\nonumber
		\\\nonumber\\
		&\stackrel{\eqref{diagL}}{=}&
		i\M(z,\z)
		\begin{pmatrix}
			-\sqrt{1+2\nu(z,\z)} & 0\\
			0 & \sqrt{1+2\nu(z,\z)}
		\end{pmatrix}
		\begin{pmatrix}
			\Lambda \eta\\
			\Lambda\eb
		\end{pmatrix}\nonumber\\\\
		&\stackrel{\eqref{Q},\eqref{phi}}{=}&
		i\M(\eta,\eb)
		\begin{pmatrix}
			-\sqrt{1+\mu\left(\phi\left(Q(\eta,\eb)\right)\right)} & 0\\
			0 & \sqrt{1+2\mu\left(\phi\left(Q(\eta,\eb\right)\right)}
		\end{pmatrix}
		\begin{pmatrix}
			\Lambda \eta\\
			\Lambda\eb
		\end{pmatrix}
		\nonumber\\\nonumber\\
		&\stackrel{\eqref{mu},\eqref{phi}}{=}&
		i\M(\eta,\eb)
		\begin{pmatrix}
			-(1-cQ(\eta,\eb)) & 0\\
			0 & 1-cQ(\eta,\eb)
		\end{pmatrix}
		\begin{pmatrix}
			\Lambda \eta\\
			\Lambda\eb
		\end{pmatrix}.\nonumber
	\end{eqnarray}
	Finally, by matching \eqref{sysS} and \eqref{sysD} one gets
	\begin{equation}\label{system5.1}
		\dt
		\begin{pmatrix}
			\eta\\
			\eb
		\end{pmatrix}
		+\M(\eta,\eb)^{-1}\dt\Big(\M(\eta,\eb)\Big)
		\begin{pmatrix}
			\eta\\
			\eb
		\end{pmatrix}
		=i(1-cQ(\eta,\eb))
		\begin{pmatrix}
			-\Lambda \eta\\
			\Lambda\eb
		\end{pmatrix}.
	\end{equation}
	
	We have to study the left-hand side of \eqref{system5.1}.
	By \eqref{Mm1} and \eqref{rho}
	$$\M(\eta,\eb)^{-1}=\frac{1}{\sqrt{1-\rh^2(\eta,\eb)}}
	\begin{pmatrix}
		1&-\rh\left(\eta,\eb\right)\\
		-\rh\left(\eta,\eb\right)&1
	\end{pmatrix},
	$$
	moreover
	$$
	\dt\Big(\M(\eta,\eb)\Big)=\frac{1}{\left(1-\rh^2(\eta,\eb)\right)^{\frac{3}{2}}}
	\begin{pmatrix}
		\rh\left(\eta,\eb)\right)&1\\
		1&\rh\left(\eta,\eb\right)
	\end{pmatrix}
	\dt\rh(\eta,\eb).
	$$
	We then have
	\begin{equation}\label{system6.1}
		\M(\eta,\eb)^{-1}\dt\Big(\M(\eta,\eb)\Big)=\frac{1}{1-\rh^2(\eta,\eb)}
		\begin{pmatrix}
			0&1\\
			1&0
		\end{pmatrix}
		\dt\rh(\eta,\eb).
	\end{equation}
	Now, keeping in mind the definition of $\rh$ in \eqref{rho22}, we have
	\begin{equation*}
		\dt\rh(\eta,\eb)=\frac{2c}{\left(2-cQ(\eta,\eb)\right)^2}\dt Q(\eta,\eb)=\frac{2c}{\left(2-cQ(\eta,\eb)\right)^2}\left\langle\Lambda(\eta+\eb),\dt\eta+\dt\eb\right\rangle
	\end{equation*}
	and
	\begin{equation*}
		1-\rh^2(\eta,\eb)=\frac{4-4cQ(\eta,\eb)}{\left(2-cQ(\eta,\eb)\right)^2}.
	\end{equation*}
	Putting these together with \eqref{system6.1}, we get
	\begin{equation*}
		\M(\eta,\eb)^{-1}\dt\left\{\M(\eta,\eb)\right\}=\frac{c}{2\left(1-cQ(\eta,\eb)\right)}
		\begin{pmatrix}
			0&1\\
			1&0
		\end{pmatrix}
		\left\langle\Lambda(\eta+\eb),\dt\eta+\dt\eb\right\rangle.
	\end{equation*}
	Let us consider the operator
	\begin{equation*}\label{kappa}
		K_1(\alpha_1,\alpha_2)
		\begin{pmatrix}
			\beta_1\\
			\beta_2
		\end{pmatrix}:=
		\frac{c}{2\left(1-cQ(\alpha_1,\alpha_2)\right)}\left\langle\Lambda(\alpha_1+\alpha_2),\beta_1+\beta_2\right\rangle
		\begin{pmatrix}
			\alpha_1\\
			\alpha_2
		\end{pmatrix}.
	\end{equation*}
	We have that system \eqref{system5.1} can be rewritten as
	\begin{equation}\label{system8.1}
		\big(Id+K_1(\eta,\eb)\big)
		\begin{pmatrix}
			\dt\eta\\
			\dt\eb
		\end{pmatrix}
		=i(1-cQ(\eta,\eb))
		\begin{pmatrix}
			-\Lambda \eta\\
			\Lambda\eb
		\end{pmatrix}.
	\end{equation}
	Using the Neumann series, we have that a formal inverse for $\Id+K_1$ is given by
	$$\left(\Id+K_1(\eta,\eb)\right)^{-1}=\Id+\sum_{n=1}^\infty(-1)^nK_1(\eta,\eb)^n$$
	provided the right-hand side series converges.
	By defining 
	$$F(\eta,\eb):=\frac{c}{2\left(1-cQ(\eta,\eb)\right)}$$
	we have
	$$K_1^n(\eta,\eb)\begin{pmatrix}
		\alpha_1\\
		\alpha_2
	\end{pmatrix}=
	\begin{pmatrix}
		\eb\\
		\eta
	\end{pmatrix}
	\left\langle\Lambda(\eta+\eb),\alpha_1+\alpha_2\right\rangle F^n(\eta,\eb)\left\langle\Lambda(\eta+\eb),\eta+\eb\right\rangle^{n-1},
	$$
	so the Neumann series converges if $\left|F(\eta,\eb)\left\langle\Lambda(\eta+\eb),\eta+\eb\right\rangle\right|<1$.
	Since we have $$F(\eta,\eb)\left\langle\Lambda(\eta+\eb),\eta+\eb\right\rangle=\dfrac{cQ(\eta,\eb)}{1-cQ(\eta,\eb)},$$
	for the series to converge, it is sufficient that $|Q(\eta,\eb)|<1/(2|c|)$ and hence
	\begin{equation*}\label{epsc2}
		(\eta,\eb)\in B_{\frac{\dc}{2}}\left(H_0^{\frac{1}{2}}(\T^n,c.c)\right).
	\end{equation*}
	Therefore we have
	\begin{eqnarray*}
		\sum_{n=1}^\infty(-1)^nK_1(\eta,\eb)^n\begin{pmatrix}
			\alpha_1\\
			\alpha_2
		\end{pmatrix}&=&
		\begin{pmatrix}
			\eb\\
			\eta
		\end{pmatrix}\left(-F(\eta,\eb)\right)\left\langle\Lambda(\eta+\eb),\alpha_1+\alpha_2\right\rangle 
		\sum_{n=0}^\infty\left(-F(\eta,\eb)\left\langle\Lambda(\eta+\eb),\eta+\eb\right\rangle\right)^n\\
		&=&\begin{pmatrix}
			\eb\\
			\eta
		\end{pmatrix}\dfrac{-F(\eta,\eb)}{1+F(\eta,\eb)\left\langle\Lambda(\eta+\eb),\eta+\eb\right\rangle}\left\langle\Lambda(\eta+\eb),\alpha_1+\alpha_2\right\rangle\\
		&=&\begin{pmatrix}
			\eb\\
			\eta
		\end{pmatrix}\Bigg[-\frac{c}{2}\left\langle\Lambda(\eta+\eb),\alpha_1+\alpha_2\right\rangle\Bigg].
	\end{eqnarray*}
	Thanks to this calculation we can rewrite system \eqref{system8.1} in the following way:
	\begin{eqnarray*}
		\dt
		\begin{pmatrix}
			\eta\\
			\eb
		\end{pmatrix}
		&=&i\left(1-cQ(\eta,\eb)\right)\left(\Id+K(\eta,\eb)\right)^{-1}
		\begin{pmatrix}
			-\Lambda\eta\\
			\Lambda\eb
		\end{pmatrix}\\
		&=&
		i\left(1-cQ(\eta,\eb)\right)
		\begin{pmatrix}
			-\Lambda\eta\\
			\Lambda\eb
		\end{pmatrix}
		-\dfrac{ic\left(1-cQ(\eta,\eb)\right)}{2}\left(\left\langle\Lambda\eb,\Lambda\eb\right\rangle-\left\langle\Lambda\eta,\Lambda\eta\right\rangle\right)
		\begin{pmatrix}
			\eb\\
			\eta
		\end{pmatrix},
	\end{eqnarray*}
	that is system \eqref{system9.1}.\\
	
	Let $(\eta,\eb)$ be a solution of \eqref{system9.1}. Then the following a priori energy estimate holds:
	
	\begin{eqnarray}\label{ee1}
		\dt\|\eta\|_s^2&=&2\Re\left(i\left(1-cQ(\eta,\eb)\right)\jap{\Lambda^{2s}\eta,\,\Lambda\eb-\dfrac{c}{2}\left(\left\langle\Lambda\eb,\Lambda\eb\right\rangle-\left\langle\Lambda\eta,\Lambda\eta\right\rangle\right)\eta}\right)\nonumber\\
		&=&2\Re\left(\dfrac{ic}{2}\left(1-cQ(\eta,\eb)\right)\left(\left\langle\Lambda\eb,\Lambda\eb\right\rangle-\left\langle\Lambda\eta,\Lambda\eta\right\rangle\right)\jap{\Lambda^{2s}\eta,\,\eta}\right)\\
		&\stackrel{\eqref{Q12}}{\leq}&|c|\|\eta\|^2_{1}\|\eta\|_s^2+c^2\|\eta\|^4_{1}\|\eta\|_s^2\nonumber.
	\end{eqnarray}
	This gives the local existence for solutions of \eqref{MK} with initial data $(\eta_0,\eb_0)\in H_0^1(\T^n,c.c.)$ in a time interval of length $T\sim \|\eta_0\|_1^{-2}$. Thanks to a bootstrap argument, the same result can be extended to $H_0^s$ with $s>1$. This lower bound for the time of existence of solutions is consistent with the one obtained by Dickey in \cite{Dickey}.\\
	Let us note moreover that the diagonal part of \eqref{system9.1}, namely  $\dfrac{ic\left(1-cQ(\eta,\eb)\right)}{2}\begin{pmatrix}
		-\Lambda\eta\\
		\Lambda\eb
	\end{pmatrix},$ does not contribute to the energy estimate \eqref{ee1}.\\
	Indeed, the first term that gives a non-trivial contribution is the off-diagonal cubic one, i.e. 
	\begin{equation}\label{B3}\B_{3}(\eta,\eb):=-\frac{ic}{2}\big(\langle\Lambda\eb,\Lambda\eb\rangle-\langle\Lambda\eta,\Lambda\eta\rangle\big)
		\begin{pmatrix}
			\eb\\
			\eta
		\end{pmatrix},
	\end{equation}
	while the remaining terms give a contribution of higher order. The next step is therefore the cancellation of $\B_3$.

	\section{Normal form: first step}\label{first step}
	The aim of this section is to remove the off-diagonal cubic term $\B_3$ by constructing a (normal form) transformation $\Phi^{(3)}$  with the aim of removing the non-resonant cubic terms. \\
	The main result we will prove is the following.
	\begin{lemma}{\bf (First step of normal form)}\label{lemmafirst}\\
		There exist $\delta_2>0$ (defined in \eqref{delta1} ) and a map 
		\begin{equation*}\Phi^{(3)}:B_{2\delta_2}(H^{m_0}_0\big(\T^n,c.c.)\big)\rightarrow B_{\delta_2}(H^{m_0}_0\big(\T^n,c.c.)\big)
		\end{equation*} that conjugates system \eqref{system9.1} to one of the form
		\begin{equation}\label{ssss}\dt(w,\w)=X^{(1)}(w,\w)=\left(\Id+\P(w,\w)\right)\D_1+X_{{\bf Res},3}(w,\w)+X^{(1)}_{\geq5}(w,\w),
		\end{equation}
		where $\P$ is a real scalar function, defined in \eqref{P}, $\D_1$ is diagonal, $X_{{\bf Res},3}(w,\w)$ consist of homogeneous cubic resonant terms and is defined in \eqref{X311}, \eqref{X312} and $X^{(1)}_{\geq5}(w,\w)$ is bounded and contains the remaining terms of homogeneity at least 5.\\
		Finally, the terms $\left(\Id+\P(w,\w)\right)\D_1$ and $X_{{\bf Res},3}(w,\w)$  give no contribution to the energy estimates, namely the Sobolev
		norms of the solutions of the system $\dt(w,\w)=\left(\Id+\P(w,\w)\right)\D_1+X_{{\bf Res},3}(w,\w)$ are constant.
	\end{lemma}
	
	We start by grouping together the terms of system \eqref{system9.1} by homogeneity:
	\begin{equation}\label{system10.1}
		\dt(\eta,\eb)=X(\eta,\eb):=\D_1(\eta,\eb)+\D_{3}(\eta,\eb)+\B_3(\eta,\eb)+\B_{5}(\eta,\eb),
	\end{equation}
	where
	\begin{equation*}\label{D1}
		\D_1(\eta,\eb)=
		\begin{pmatrix}
			-i\Lambda\eta\\
			i\Lambda\eb
		\end{pmatrix}
	\end{equation*} 
	is the linear diagonal part of \eqref{system9.1},
	\begin{equation*}\label{D3}
		\D_{3}(\eta,\eb)=-cQ(\eta,\eb) \D_1(\eta,\eb)
	\end{equation*}
	is its diagonal cubic part, $\B_3$ is the cubic off-diagonal part defined in \eqref{B3} and, 
	\begin{equation}\label{R5}
		\B_{5}(\eta,\eb)=-cQ(\eta,\eb)\B_{3}(\eta,\eb)
	\end{equation}
	is the remaining quintic part.\\
	Let us formally define the map 
	\begin{equation}\label{change4}
		\begin{pmatrix}
			\eta\\
			\eb
		\end{pmatrix}
		=\Phi^{(3)}(w,\w):=(\Id+M(w,\w))
		\begin{pmatrix}
			w\\
			\w
		\end{pmatrix},
	\end{equation}
	where $M$ is a bi-linear map with values in the space of $2\times2$ matrices, more precisely
	\begin{equation*}
		M(w_1,w_2)=
		\begin{pmatrix}
			M_{11}(w_1,w_2)&M_{12}(w_1,w_2)\\
			M_{21}(w_1,w_2)&M_{22}(w_1,w_2)
		\end{pmatrix}
	\end{equation*}
	with $M_{ij}(w_1,w_2)=A_{ij}[w_1,w_1]+B_{ij}[w_1,w_2]+C_{ij}[w_2,w_2]$. We moreover denote 
	$$A[w,w]:=\begin{pmatrix}
		A_{11}\big[w^{(1)}_1,w_1^{(2)}\big]&A_{12}\big[w_1^{(1)},w_1^{(2)}\big]\\
		A_{21}\big[w_1^{(1)},w_1^{(2)}\big]&A_{22}\big[w_1^{(1)},w_1^{(2)}\big]
	\end{pmatrix},$$
	where each component $A_{\ell_1,\ell_2}$ is an operator that acts at the level of  Fourier series $w=\sum_{k\in\Z^n\backslash\{0\}}w_je^{i j\cdot x},\,h=\sum_{k\in\Z^n\backslash\{0\}}h_ke^{ik\cdot x}$ in the following way:
	\begin{equation}\label{A12}
		A_{\ell_1,\ell_2}[w,w]h:=\sum_{j,k\in\Z^n\backslash\{0\}}a_{\ell_1,\ell_2}(j,k)w_jw_{-j}h_ke^{ik\cdot x}.
	\end{equation} 
	Similarly for $B$ and $C$.\\
	We assume $A$ and $C$ to be symmetric, i.e.
	$$A\big[w^{(1)}_1,w_1^{(2)}\big]=A\big[w^{(2)}_1,w_1^{(1)}\big],\quad C\big[w^{(1)}_2,w_2^{(2)}\big]=C\big[w^{(2)}_2,w_2^{(1)}\big]\quad \text{for all }w_1^{(1)},\,w^{(2)}_1,\,w^{(1)}_2,\,w^{(2)}_2.$$
	Moreover, in order for $\Phi^{(3)}$ to be a well-posed map from complex-conjugated functions, we impose
	\begin{equation*}\label{misibevon}
		\overline{M_{11}(w_1,w_2)}=M_{22}(w_1,w_2),\quad \overline{M_{12}(w_1,w_2)}=M_{21}(w_1,w_2).
	\end{equation*}
	\\
	Let us evaluate how $\Phi^{(3)}$ transforms system \eqref{system10.1}: we have
	\begin{eqnarray}\label{miserve1}
		\dt
		\begin{pmatrix}
			\eta\\
			\eb
		\end{pmatrix}
		&=&\dt\Phi^{(3)}(w,\w)=\dt\left[(\Id+M(w,\w))
		\begin{pmatrix}
			w\\
			\w
		\end{pmatrix}\right]\nonumber\\\\
		&=&\left(\Id+M(w,\w)\right)\dt
		\begin{pmatrix}
			w\\
			\w
		\end{pmatrix}
		+\dt\left\{M(w,\w)\right\}
		\begin{pmatrix}
			w\\
			\w
		\end{pmatrix}.\nonumber
	\end{eqnarray}
	
	Since $A$ and $C$ are symmetric, one has
	$$\dt\left\{M(w,\w)\right\}
	\begin{pmatrix}
		w\\
		\w
	\end{pmatrix}=\left[2A[w,w_t]+B[w_t,\w]+B[w,\w_t]+2C[\w,w_t]\right]
	\begin{pmatrix}
		w\\
		\w
	\end{pmatrix},$$
	then, by defining 
	\begin{equation*}\label{K21}
		K_2(w,\w)
		\begin{pmatrix}
			\alpha\\
			\beta
		\end{pmatrix}
		:=
		M(w,\w)
		\begin{pmatrix}
			\alpha\\
			\beta
		\end{pmatrix}+M_t(w,\w)
		\begin{pmatrix}
			\alpha\\
			\beta
		\end{pmatrix},
	\end{equation*}
	with
	\begin{equation*}\label{Mt}
		M_t(w,\w)
		\begin{pmatrix}
			\alpha\\
			\beta
		\end{pmatrix}:=
		\left[2A[w,\alpha]+B[\alpha,\w]+B[w,\beta]+2C[\w,\beta]\right]
		\begin{pmatrix}
			w\\
			\w
		\end{pmatrix},
	\end{equation*}
	we have that \eqref{miserve1} can be written as
	\begin{equation*}\label{system11.1}
		\dt
		\begin{pmatrix}
			\eta\\
			\eb
		\end{pmatrix}=
		(\Id+K_2(w,\w))\begin{pmatrix}
			\dt w\\
			\dt\w
		\end{pmatrix}.
	\end{equation*}
	
	On the other hand, by \eqref{change4}
	\begin{eqnarray*}\label{miserve2}
		\dt\begin{pmatrix}
			\eta\\
			\eb
		\end{pmatrix}&=&X\left(\Phi^{(3)}(w,\w)\right)\nonumber\\\\
		&=&\D_1\left(\Phi^{(3)}(w,\w)\right)+\D_{3}\left(\Phi^{(3)}(w,\w)\right)+\B_3\left(\Phi^{(3)}(w,\w)\right)+\B_{5}\left(\Phi^{(3)}(w,\w)\right).\nonumber
	\end{eqnarray*}
	We arrive to
	\begin{equation*}\label{miserve3}
		\dt\begin{pmatrix}
			w\\
			\w
		\end{pmatrix}=X^{(1)}(w,\w):=\left(\Id+K_2(w,\w)\right)^{-1}\left(X\left(\Phi^{(3)}(w,\w)\right)\right).
	\end{equation*}
	Therefore, in order to derive the equation for the variables $(w,\w)$ one must be able to invert the operator $\Id+K_2$:\\
	using the Neumann series we (formally) have 
	$$\left(\Id+K_2(w,\w)\right)^{-1}=\Id-K_2(w,\w)+\tilde{K}_2(w,\w),$$ where
	\begin{equation*}\label{TK}
		\tilde{K}_2(w,\w)=\sum_{n=2}^\infty(-1)^nK_2(w,\w)^n.
	\end{equation*}
	We then have
	\begin{equation*}\label{X1v}
		X^{(1)}(w,\w)=\D_1(w,\w)+\tilde{D}_{\geq3}(w,\w)+X^{(1)}_3(w,\w)+X^{(1)}_{\geq5}(w,\w),
	\end{equation*}
	where
	\begin{itemize}
		\item[{\bf (i)}]The diagonal part of degree 3 is
		\begin{equation*}\label{td3}
			\tilde{\D}_{\geq3}(w,\w):=\P(w,\w)\D_1(w,\w)
		\end{equation*}
		with
		\begin{equation}\label{P}
			\P(w,\w):=-cQ\left(\Phi^{(3)}(w,\w)\right).
		\end{equation}
		
		\item[{\bf (ii)}]The non-diagonal cubic terms of $X^{(1)}$ are 
		\begin{equation}\label{Eq4.0}
			X_3^{(1)}(w,\w):=\D_1\left(M(w,\w)
			\begin{pmatrix}
				w\\
				\w
			\end{pmatrix}
			\right)-K_2(w,\w)\D_1(w,\w)+\B_3(w,\w),
		\end{equation}
		
		\item[{\bf(iii)}]Finally, the remaining terms with degree at leas $5$ are
		\begin{eqnarray}\label{xgeq5}
			X^{(1)}_{\geq5}(w,\w)&=&-K_2(w,\w)\D_1\left(M(w,\w)
			\begin{pmatrix}
				w\\
				\w
			\end{pmatrix}
			\right)+\tilde{K}_2(w,\w)\D_1\left(\Phi^{(3)}(w,\w)\right)\nonumber\\
			&&+\P(w,\w)\D_1\left(M(w,\w)
			\begin{pmatrix}
				w\\
				\w
			\end{pmatrix}\right)+\left(K_2(w,\w)+\tilde{K}_2(w,\w) \right)\D_{3}\left(\Phi^{(3)}(w,\w)\right)\nonumber\\
			&&+\left(\Id+K_2(w,\w)\right)^{-1}\B_{5}\left(\Phi^{(3)}(w,\w)\right)+\left[\B_3\left(\Phi^{(3)}(w,\w)\right)-\B_3(w,\w)\right]\nonumber\\
			&&+\left(-K_2(w,\w)+\tilde{K}_2(w,\w)\right)\B_3\left(\Phi^{(3)}(w,\w)\right).
		\end{eqnarray}
	\end{itemize}
	Let us analyze the first component of the cubic vector field  $X_3^{(1)}(w,\w)$:
	\begin{eqnarray*}
		\left(X_3^{(1)}(w,\w)\right)_1&=&-i\Lambda M_{11}(w,\w)w -i\Lambda M_{12}(w,\w)\w +iM_{11}(w,\w)\Lambda w-iM_{12}(w,\w)\Lambda \w\\
		&&-i\left\{-2A_{11}(w,\Lambda w)-B_{11}(\Lambda w,\w)+B_{11}(w,\Lambda \w)+2C_{11}(\w,\Lambda\w)\right\}w\\
		&&-i\left\{-2A_{11}(w,\Lambda w)-B_{11}(\Lambda w,\w)+B_{11}(w,\Lambda \w)+2C_{11}(\w,\Lambda \w)\right\}\w\\
		&&-\frac{ic}{2}\left(\left\langle\Lambda \w,\Lambda \w\right\rangle-\left\langle\Lambda w,\Lambda w\right\rangle\right)\w.
	\end{eqnarray*}
	Since our aim is to choose the coefficients $M_{i,j}$ to eliminate as many terms as possible from $X^{(1)}_3$, we set 
	$$M_{11},B_{12}=0.$$ Under these assumptions it remains
	\begin{eqnarray*}
		\left(X_3^{(1)}(w,\w)\right)_1&=&-i\Lambda A_{12}[w,w]\w-i\Lambda C_{12}[\w,\w]\w-iA_{12}[w,w]\Lambda \w-iC_{12}[\w,\w]\Lambda \w\\
		&&+2iA_{12}[w,\Lambda w]\w-2iC_{12}[\w,\Lambda \w]\w-\frac{ic}{2}\left(\left\langle\Lambda \w,\Lambda \w\right\rangle-\left\langle\Lambda w,\Lambda w\right\rangle\right)\w.
	\end{eqnarray*}
	\\
	By using \eqref{A12}, we write all the quantities involved in $\left(X_3^{(1)}(w,z)\right)_1$ in Fourier series and obtain
	\begin{eqnarray*}
		\left(X_3^{(1)}(w,\w)\right)_1&=&\sum_{j,k\neq0}w_jw_{-j}\w_{-k}e^{ik\cdot x}\left[2i(|j|-|k|)a_{12}(j,k)+\frac{ic}{2}|j|^2\right]\\
		&&+\sum_{j,k\neq0}\w_j\w_{-j}\w_{-k}e^{ik\cdot x}\left[-2i(|j|+|k|)c_{12}(j,k)-\frac{ic}{2}|j|^2\right] .
	\end{eqnarray*}
	In order to cancel out as many terms as possible we set
	\begin{equation}\label{ac}
		a_{12}(j,k)=
		\left\{
		\begin{aligned}
			\frac{c|j|^2}{4(|k|-|j|)}\quad\text{if }j\neq k\\
			0\quad\quad\quad\quad\text{if }j= k
		\end{aligned},
		\right.
		\quad c_{12}(j,k)=-\frac{c|j|^2}{4(|j|+|k|)}.
	\end{equation}
	By defining $X_{{\bf Res}, 3}:=X^{(1)}_3$, from \eqref{ac} we have that
	\begin{equation}\label{X311}
		\left(X_{{\bf Res}, 3}(w,\w)\right)_1=\frac{ic}{2}\sum_{j\,k\neq0\,|j|=|k|}|j|^2w_jw_{-j}\w_{-k}e^{ik\cdot x}.
	\end{equation}
	Similar calculations for $\left(X^{(1)}(w,z)\right)_2$ yield
	\begin{equation*}\label{ac2}
		M_{22}=0,\quad B_{21}=0,\quad A_{21}=C_{12},\quad C_{21}=A_{12}
	\end{equation*}
	and 
	\begin{equation}\label{X312}
		\left(X_{{\bf Res}, 3}(w,\w)\right)_2=-\frac{ic}{2}\sum_{j\,k\neq0\,|j|=|k|}|j|^2\w_j\w_{-j}w_ke^{ik\cdot x}.
	\end{equation}
	\\
	\begin{remark}\label{bh1}
		The denominator $||k|-|j||$ of \eqref{ac} is the one responsible for the different regularity thresholds $m_0$ in \eqref{m0}. In fact, in dimension $n=1$, one has $||j|-|k||\geq1$ for all $|j|\neq|k|$ and therefore
		\begin{eqnarray}\label{stimA1}
			\|A_{1,2}[w,w]\w\|^2_s&=&\sum_{k\neq0}\left|\sum_{j\neq0,|j|\neq|k|}w_jw_{-j}\dfrac{c|j|^2}{4(|k|-|j|)}\w_k\right|^2|k|^{2s}\nonumber\\\\
			&\leq&\dfrac{|c|}{16}\sum_{k\neq0}\left(\sum_{j\neq0,|j|\neq|k|}|j|^2|w_j|^2\right)^2|w_k|^2|k|^{2s}\leq\dfrac{1}{16}\|w\|_{1}^4\|w\|_s^2\nonumber.
		\end{eqnarray}
		In dimension $n\geq2$, on the other hand, the best possible bound is
		\begin{equation}\label{trick}
			\dfrac{1}{||k|-|j||}\leq\max\{3|j|,3|k|\}\quad\forall j,k\in\Z^n\backslash\{0\},\,\,|j|\neq|k|.
		\end{equation}
		Indeed, if $||j|-|k||>1$, \eqref{trick} follows trivially, while if $||j-|k||\leq1$ one has $|j|\leq|k|+1$ or $|k|\leq|j|+1$ and 
		$$\dfrac{1}{||j|-|k||}\leq\dfrac{|j|+|k|}{\left||j|^2-|k|^2\right|}\leq\max\{3|j|,3|k|\}.$$
		Thanks to \eqref{trick}, a computation similar to the one performed in \eqref{stimA1} yields 
		\begin{equation}\label{stimA2}
			\|A_{1,2}[w,w]\w\|^2_s\leq\dfrac{9|c|}{16}\|w\|_{\frac{3}{2}}^4\|w\|_s^2
		\end{equation}
		for the case $n\geq2$.
	\end{remark}
	\bigskip
	We can therefore explicitly write the operator matrices
	\begin{equation*}\label{M}
		M(w,\z)
		\begin{pmatrix}
			\alpha\\
			\beta
		\end{pmatrix}
		=
		\begin{pmatrix}
			\left(A_{12}[w,w]+C_{12}[\w,\w]\right)\beta\\
			\left(A_{21}[w,w]+C_{21}[\w,\w]\right)\alpha
		\end{pmatrix}
	\end{equation*}
	and
	\begin{equation}\label{K2}
		K_2(w,\w)
		\begin{pmatrix}
			\alpha\\
			\beta
		\end{pmatrix}=
		\begin{pmatrix}
			A_{12}[w,w]\beta+C_{12}[\w,\w]\beta+2A_{12}[w,\alpha]\w+2C_{12}[\w,\beta]\w\\
			A_{21}[w,w]\alpha+C_{21}[\w,\w]\alpha+2A_{21}[w,\alpha]w+2C_{21}[\w,\beta]w
		\end{pmatrix}.
	\end{equation}
	\bigskip
	It holds the following quantitative result:
	\begin{lemma}
		Let $m_0$ be as in \eqref{m0}, for every triple of complex functions $u,\,v,\,h$ and all real $s\geq0$ one has
		\begin{equation*}\label{estim1}
			\|A_{12}[u,v]h\|_s\leq\frac{3|c|}{4}\|u\|_{m_0}\|v\|_{m_0}\|h\|_s,\quad\quad \|C_{12}[u,v]h\|_s\leq\frac{|c|}{8}\|u\|_1\|v\|_1\|h\|_s.
		\end{equation*}
	\end{lemma}
	\begin{proof}
		This follows directly by performing the same computations of Remark \ref{bh1}.
	\end{proof}
	\begin{lemma}
		For all $s\geq0$, all $(w,\w)\in H^{m_0}_0(\T^n,\,c.c.)$, $(\alpha, \bar{\alpha})\in H^s_0(\T^n,\,c.c.)$ one has
		\begin{equation}\label{estimM}
			\left\|M(w,\w)
			\begin{pmatrix}
				\alpha\\
				\bar{\alpha}
			\end{pmatrix}\right\|_s\leq\frac{7|c|}{8}\|w\|^2_{m_0}\|\alpha\|_s;
		\end{equation}
		\begin{equation}\label{estimK}
			\left\|K_2(w,\w)
			\begin{pmatrix}
				\alpha\\
				\bar{\alpha}
			\end{pmatrix}\right\|_s\leq\frac{7|c|}{8}\|w\|^2_{m_0}\|\alpha\|_s+\frac{7|c|}{4}\|w\|_{m_0}\|w\|_s\|\alpha\|_{m_0}.\\
		\end{equation}
		\bigskip
		\\
		Finally, there exists $\delta_1>0$, such that, if  $\|w\|_{m_0}<\delta_1$ the operator 
		$$\Id+K(w,\w): H_0^{m_0}(\T^n,c.c.)\longrightarrow H_0^{m_0}(\T^n,c.c.)$$
		is invertible  with inverse satisfying
		\begin{equation}\label{inverse}
			\left\|\left(\Id+K_2(w,\w)\right)^{-1}
			\begin{pmatrix}
				\alpha\\
				\bar{\alpha}
			\end{pmatrix}
			\right\|_s\leq C\left(\|\alpha\|_s+\|w\|_ {m_0}\|w\|_ s\|\alpha\|_ {m_0}\right)\quad\text{ for every }s\geq m_0,
		\end{equation}
		where $C$ depends on $c$.
	\end{lemma}
	\begin{proof}
		From \eqref{estimK} it is clear that, for a fixed $(w,\w)\in H_0^{m_0}(\T^n,c.c.),$ $\Id+K(w,\w)$ is a bounded operator from $ H_0^{m_0}(\T^n,c.c.)$ into itself. Having that 
		$$\left\|K_2(w,\w)
		\begin{pmatrix}
			\alpha\\
			\bar{\alpha}
		\end{pmatrix}
		\right\|_{m_0}\leq\frac{21|c|}{8}\|w\|_{m_0}^2\|\alpha\|_{m_0},
		$$
		Now, using the Neumann series we can compute explicitly that
		$$\left\|\left(\Id+K_2(w,\w)\right)^{-1}
		\begin{pmatrix}
			\alpha\\
			\bar{\alpha}
		\end{pmatrix}
		\right\|_s=\left\|\left(\Id+\sum_{n=1}^{\infty}(-1)^nK_2^n(w,\w)\right)
		\begin{pmatrix}
			\alpha\\
			\bar{\alpha}
		\end{pmatrix}
		\right\|_s\leq\left\|
		\alpha
		\right\|_s+\sum_{n=1}^{\infty}\left\|K_2^n(w,\w)
		\begin{pmatrix}
			\alpha\\
			\bar{\alpha}
		\end{pmatrix}
		\right\|_s.
		$$
		
		Finally, by induction one can verify that
		$$
		\left\|K_2^n(w,z)
		\begin{pmatrix}
			\alpha\\
			\bar{\alpha}
		\end{pmatrix}
		\right\|_s\leq\Bigg(\frac{7|c|}{8}\|w\|^2_{m_0}\Bigg)^n\|\alpha\|_s+\Bigg(\frac{21|c|}{8}\|w\|_{m_0}^2\Bigg)^{n-1}\frac{7|c|}{4}\|w\|_{m_0}\|w\|_s\|\alpha\|_{m_0}$$
		$$+\Bigg[\sum_{j=1}^{n-1}\Bigg(\frac{7|c|}{8}\|w\|_{m_0}^2\Bigg)^j\Bigg(\frac{21|c|}{8}\|w\|_{m_0}^2\Bigg)^{n-1-j}\Bigg]\frac{7|c|}{4}\|w\|_{m_0}\|w\|_s\|\alpha\|_{m_0},
		$$
		so, imposing for instance 
		\begin{equation}\label{delta1}
			\|w\|_{m_0}<\delta_1:=\sqrt{4/(21|c|)}
		\end{equation}
		we can derive \eqref{inverse}
	\end{proof}
	By using the same contraction argument as in \cite[Lemma 4.3]{BaldiHaus1}, we can also prove that the map $\Phi^{(3)}$ is invertible near the origin:
	\begin{lemma}\label{contract1}
		There exists $\delta_2>0$, such that for all $(\eta,\eb)\in  H_0^{m_0}(\T^n,c.c.)$ in the ball 
		\begin{equation*}\label{condpsi4}
			\|\eta\|_{m_0}\leq\delta_2,
		\end{equation*}
		there exists a unique $(w,\w)\in  H_0^{m_0}(\T^n,c.c.)$ such that $\Phi^{(3)}(w,\w)=(\eta,\eb)$, with $\|(w,\w)\|_{m_0}\leq2\|(\eta,\eb)\|_{m_0}$.\\
		If, in addition, $\eta\in H_0^{s}(\T^n,c.c.)$ for some $s>m_0$, then $w$ also belongs to $H^s_0$, and 
		\begin{equation}\label{bond3}
		\|w\|_s\leq2\|\eta\|_s.
		\end{equation}
	\end{lemma}
	\begin{proof}
		Following the idea of the original proof, for a chosen element $(\eta,\eb)\in H_0^{m_0}(\T^n,c.c.)$ we want to find a fixed point for the map
		$$\Psi(w,\w):=
		\begin{pmatrix}
			\eta\\
			\eb
		\end{pmatrix}
		-M(w,\w)
		\begin{pmatrix}
			w\\
			\w
		\end{pmatrix}
		.
		$$

		Let us consider the ball $B_R:=B_R\left(H_0^{m_0}(\T^n,c.c.)\right)$, a direct application of \eqref{estimM} shows that $\Psi$ maps $B_R$ into itself if 
		\begin{equation}\label{cond4.1}
			\|\eta\|_{m_0}+\frac{7|c|}{8}R^3\leq R.
		\end{equation}
		It remains to check if $\Psi$ is a contraction:
		$$\|\Psi(w_1,\w_1)-\Psi(w_2,\w_2)\|_{m_0}=
		\left\|M(w_1,\w_1)
		\begin{pmatrix}
			w_1\\
			\w_1
		\end{pmatrix}
		-M(w_2,\w_2)
		\begin{pmatrix}
			w_2\\
			\w_2
		\end{pmatrix}
		\right\|_{m_0}
		$$
		$$\leq\int_0^1\left\|K\big(w_2+\theta(w_1-w_2),\w_2+\theta(\w_1-\w_2)\big)
		\begin{pmatrix}
			w_1-w_2\\
			\w_1-\w_2
		\end{pmatrix}
		\right\|_{m_0}d\theta\leq\frac{21|c|}{8}R^2\|w_1-w_2\|_{m_0}.
		$$
		Hence $\Psi$ is a contraction if 
		\begin{equation}\label{cond4.2}
			\frac{21|c|}{8}R^2<1.
		\end{equation}
		\bigskip
		\\
		By choosing $R=2\|\eta\|_{m_0}$ and using \eqref{cond4.1} together with \eqref{cond4.2} we have that $\Psi$ is a contraction in the ball $B_R$ if  $\|\eta\|_{m_0}\leq\delta_2$, with
		\begin{equation*}\label{delta2}
			\delta_2=(21|c|)^{-\frac{1}{2}}
		\end{equation*}
	\end{proof}
	
	\begin{lemma}\label{brutto1}
		For all complex functions $u,\,v,\,h,\,y$ one has
		\begin{align}
			\langle A_{12}[u, v]y, h \rangle &= \langle y, A_{12}[u, v]h \rangle, & \langle C_{12}[u, v]y, h \rangle &= \langle y, C_{12}[u, v]h \rangle\nonumber , \\
			\overline{A_{12}[u, v]y} &= A_{12}[\bar{u}, \bar{v}]\bar{y}, & \overline{C_{12}[u, v]y} &= C_{12}[\bar{u}, \bar{v}]\bar{y}\nonumber ,\\
			A_{12}[u, v]\Lambda^s y &= \Lambda^s A_{12}[u, v]y, & C_{12}[u, v]\Lambda^s y &= \Lambda^s C_{12}[u, v]y\nonumber 
		\end{align}
		From these relations it follows that
		\begin{align}
			\langle M_{12}(u, v)y, h \rangle &= \langle y, M_{12}(u, v)h \rangle, & \langle M_{21}(u, v)y, h \rangle &= \langle y, M_{21}(u, v)h \rangle,\nonumber \\
			\overline{M_{12}(u, v)h} &= M_{12}(\bar{u}, \bar{v})\bar{h}, & \overline{M_{21}(u, v)h} &= M_{21}(\bar{u}, \bar{v})\bar{h},\nonumber  \\
			[M_{12}(u, v), \Lambda^s] &= 0, & [M_{21}(u, v), \Lambda^s] &= 0.\nonumber 
		\end{align}
		\begin{equation*}\label{br7}
			M_{12}(u, v)h = M_{21}(v, u)h
		\end{equation*}
		and
		\begin{equation}\label{Eq4.1}
			M(u, v)\mathcal{D}_1 + \mathcal{D}_1 M(u, v) = 0. 
		\end{equation}
	\end{lemma}
	\begin{proof}
		It follows directly from the definition of $A$ and $C$ in  \eqref{ac}
	\end{proof}
	Lemma \ref{brutto1} implies the reality structure of the vector field $X^{(1)}$:
	\begin{lemma}
		The vector field $X^{(1)}$ preserves the real structure \eqref{rs}, i.e., $\left(X^{(1)}\right)_1=\overline{\left(X^{(1)}\right)_2}$.
	\end{lemma}

	Let us now analyze $X^{(1)}_{\geq5}$, defined in \eqref{xgeq5}. We have:
	\begin{eqnarray*}
		X^{(1)}_{\geq5}(w,\w)&=&-K_2(w,\w)\D_1\left(M(w,\w)
		\begin{pmatrix}
			w\\
			\w
		\end{pmatrix}
		\right)+\tilde{K}_2(w,\w)\D_1\left(\Phi^{(3)}(w,\w)\right)\nonumber\\
		&&+\P(w,\w)\D_1\left(M(w,\w)
		\begin{pmatrix}
			w\\
			\w
		\end{pmatrix}\right)+\left(K_2(w,\w)+\tilde{K}_2(w,\w) \right)\D_{3}\left(\Phi^{(3)}(w,\w)\right)\nonumber\\
		&&+\tilde{X}^{(1)}_{\geq5}(w,\w),
	\end{eqnarray*}
	where
	\begin{eqnarray*} \label{txgeq5}
		\tilde{X}^{(1)}_{\geq5}(w,\w)&:=&\left(\Id+K_2(w,\w)\right)^{-1}\B_{5}\left(\Phi^{(3)}(w,\w)\right)\\
		&&+\left[\B_3\left(\Phi^{(3)}(w,\w)\right)-\B_3(w,\w)\right]+\left(-K_2(w,\w)+\tilde{K}_2(w,\w)\right)\B_3\left(\Phi^{(3)}(w,\w)\right).\nonumber
	\end{eqnarray*}
	Now, from \eqref{Eq4.1} and \eqref{Eq4.0} we have 
	\begin{equation}\label{Eq4.2}
		\left(M(w,\w)+K_2(w,\w)\right)\D_1
		\begin{pmatrix}
			w\\
			\w
		\end{pmatrix}=
		\B_3(w,\w)-X_{{\bf Res},3}(w,\w),
	\end{equation}
	hence it follows that
	\begin{eqnarray}
		&&-K_2(w,\w)\D_1\left(M(w,\w)\right)+\tilde{K}_2(w,\w)\D_1\left(\Id+M(w,\w)\right)\nonumber\\
		&=&-K_2(w,\w)\D_1\left(M(w,\w)\right)+\sum_{n=2}^\infty(-1)^nK_2^n(w,\w)\D_1\left(\Id+M(w,\w)\right)\nonumber\\
		&=&K_2(w,\w)M(w,\w)\D_1+\sum_{n=2}^\infty(-1)^nK_2^n(w,\w)\D_1-\sum_{n=2}^\infty(-1)^nK_2^n(w,\w)M(w,\w)\D_1\nonumber\\
		&=&-\sum_{n=1}^\infty(-1)^nK_2^n(w,\w)K_2(w,\w)\D_1-\sum_{n=1}^\infty(-1)^nK_2^n(w,\w)M(w,\w)\D_1\nonumber\\
		&=&-\sum_{n=1}^\infty(-1)^nK_2^n(w,\w)\left(K_2(w,\w)+M(w,\w)\right)\D_1\nonumber\\
		\label{Eq5}
		&=&K_2(w,\w)\left(\Id+K_2(w,\w)\right)^{-1}\left(\B_3(w,\w)-X_{{\bf Res},3}(w,\w)\right).
	\end{eqnarray}
	\bigskip\\
	On the other hand by \eqref{Eq4.1} and \eqref{Eq4.2} we have
	\begin{eqnarray}
		&&\P(w,\w)\D_1\left(M(w,\w)
		\begin{pmatrix}
			w\\
			\w
		\end{pmatrix}\right)+\left(K_2(w,\w)+\tilde{K}_2(w,\w) \right)\D_{3}\left(\Phi^{(3)}(w,\w)\right)\nonumber\\
		&=&\P(w,\w)\sum_{n=1}^\infty(-1)^nK^n(w,\w)\D_1+\P(w,\w)\sum_{n=0}^\infty(-1)^nK^n(w,\w)\D_1\left(M(w,\w)\right)\nonumber\\
		&=&-\sum_{n=0}^\infty(-1)^nK^n(w,\w)K_2(w,\w)\D_1-\P(w,\w)\sum_{n=0}^\infty(-1)^nK^n(w,\w)M(w,\w)\D_1\nonumber\\
		\label{Eq6}
		&=&-\P(w,\w)\left(\Id+K_2(w,\w)\right)^{-1}\left(\B_3(w,\w)-X_{{\bf Res},3}(w,\w)\right).
	\end{eqnarray}
	Now, putting together \eqref{Eq5} and \eqref{Eq6} we get
	\begin{equation}\label{X1}
		X^{(1)}(w,\w)=\left(\Id+\P(w,\w)\right)\D_1+X_{{\bf Res},3}(w,\w)+X^{(1)}_{\geq5}(w,\w),
	\end{equation}
	with
	\begin{equation}\label{xgeq5final}
		X^{(1)}_{\geq5}(w,\w)=\left(K_2(w,\w)-\P(w,\w)\right)\left(\Id+K_2(w,\w)\right)^{-1}\left(\B_3(w,\w)-X_{{\bf Res},3}(w,\w)\right)+\tilde{X}^{(1)}_{\geq5}(w,\w).
	\end{equation}
	\bigskip\\
	
	\begin{lemma}\label{su}
		For all $s\geq0$ and for any pair of complex conjugate functions $(w,\w)$ we have:
		\begin{equation}\label{B3X3}
			\left\|\B_3(w,\w)\right\|_s\leq|c|\|w\|^2_1\|w\|_s,\quad\left\|X_{{\bf Res},3}(w,\w)\right\|_s\leq\frac{|c|}{2}\|w\|_1^2\|w\|_s,
		\end{equation}	
		moreover, if $\|w\|_{m_0}\leq\delta_0$ and for all complex functions $h$, we have
		\begin{eqnarray}
			\label{ph}
			\left\|\P(w,\w)h\right\|_s&=&\left|\P(w,\w)\right|\|h\|_s,\quad0\leq\left|\P(w,\w)\right|\leq3|c|\|w\|^2_{\frac{1}{2}}\\
			\label{r5}
			\left\|\B_{5}(w,\w)\right\|_s&\leq&2c^2\|w\|^2_{\frac{1}{2}}\|w\|_1^2\|w\|_s.
		\end{eqnarray}
		
	\end{lemma}
	\begin{proof}
		Estimate \eqref{B3X3} follows from the definition of $\B_s$ and $X_{{\bf Res},3}$ in \eqref{B3}, \eqref{X311}, \eqref{X312}. Estimate \eqref{ph} follows from \eqref{P}, while estimate can be derived from \eqref{ph} together with \eqref{R5}.
	\end{proof}
	
	\begin{lemma}\label{su2}
		For all $s\geq0$, all $(w,\w)\in  H^s_0(\T^n,c.c.)\cap H^{m_0}_0(\T^n,c.c.)$ with $\|w\|_{m_0}\leq\delta_0$ we have
		\begin{equation}\label{stimx5}
			\left\|X^{(1)}_{\geq5}(w,\w)\right\|_s\leq|c|^2\|w\|^2_{m_0}\|w\|^2_1\|w\|_s.
		\end{equation}
		
	\end{lemma}
	\begin{proof}
		It follows from the expression of $X^{(1)}_{\geq5}(w,\w)$ in \eqref{xgeq5final}, together with the estimates contained in Lemma \ref{su}
	\end{proof}

	Let us now compute the contribution of the term $X_{{\bf Res}, 3}$ in the energy estimate:
	\begin{eqnarray}\label{cancellazione}
		&&2\Re\left(\jap{\Lambda^s w,\,\Lambda^2X_{{\bf Res}, 3}}\right)\nonumber\\
		&=&\jap{\Lambda^sw}{\left(\Lambda^sX_{{\bf Res}, 3}\right)_2}+\jap{\left(\Lambda^sX_{{\bf Res}, 3}\right)_1,\,\Lambda^s\w}\\
		&=&\frac{ic}{2}\sum_{j\,k\neq0\,|j|=|k|}|j|^2w_jw_{-j}\w_{k}\w_{-k}-\frac{ic}{2}\sum_{j\,k\neq0\,|j|=|k|}|j|^2\w_j\w_{-j}w_kw_{-k}=0\nonumber
	\end{eqnarray}
	and then
	\begin{eqnarray}\label{energyestimate1}
		\dt\|w\|_s^2&\stackrel{\eqref{X1}}{=}&2\Re\left(\jap{\Lambda^s w,\,\Lambda^s\left(\left(\Id+\P(w,\w)\right)\D_1+X_{{\bf Res},3}(w,\w)+X^{(1)}_{\geq5}(w,\w)\right)}\right)\nonumber\\
		&\stackrel{\eqref{cancellazione}}{=}&2\Re\left(\jap{\Lambda^s w,\,\Lambda^sX^{(1)}_{\geq5}(w,\w)}\right)\nonumber\\
		&\stackrel{\eqref{stimx5}}{\leq}&2C|c|^2\|w\|^2_{m_0}\|w\|^2_1\|w\|^2_s\nonumber
	\end{eqnarray}

	\section{Normal Form: Second Step}\label{secondstep}
	As in the previous section, the goal is to construct a normal form transformation in order to eliminate the non-resonant degree-5 terms. The main result is the following.
	\begin{lemma}{\bf (Second Step of Normal Form)}\label{lemmasecond}\\
		There exists $\delta_4>0$ (defined in \eqref{delta3} ) and a map 
		$$\Phi^{(4)}:B_{2\delta_4}(H^{m_0}_0\big(\T^n,c.c.)\big)\rightarrow B_{\delta_4}(H^{m_0}_0\big(\T^n,c.c.)\big)$$ that conjugates system \eqref{ssss} to a system of the form
		\begin{equation*}\label{sssss}
			\dt(q,\qb)=\Bigg(1+\P\left(\Phi^{(4)}(q,\qb)\right)\Bigg)\left(\D_1(q,\qb)+X_{{\bf Res},3}(q,\qb)\right)+X_{{\bf Res},5}(q,\qb)+X^{(2)}_{\geq7}(q,\qb),
		\end{equation*}
		where $X_{{\bf Res},5}(q,\qb)$ consists of homogeneity-5 resonant terms and is defined in \eqref{W51}, \eqref{W52} and $X^{(2)}_{\geq7}(w,\w)$ is bounded and contains the remaining terms of homogeneity greater than 7.\\
		Finally, the terms $\Bigg(1+\P\left(\Phi^{(4)}(q,\qb)\right)\Bigg)\left(\D_1(q,\qb)+X_{{\bf Res},3}(q,\qb)\right)$ and $X_{{\bf Res},5}(q,\qb)$ do not contribute to the energy estimates, namely the Sobolev
		norms of the solutions of the system $$\dt(q,\qb)=\Bigg(1+\P\left(\Phi^{(4)}(q,\qb)\right)\Bigg)\left(\D_1(q,\qb)+X_{{\bf Res},3}(q,\qb)\right)+X_{{\bf Res},5}(q,\qb)$$ are constant.
	\end{lemma}
	We start by splitting 
	$$X^{(1)}_{\geq5}=\P X_{{\bf Res},3}+X^{(1)}_5+X^{(1)}_{\geq7}$$
	where $\P X_{{\bf Res},3}$ does not contribute to the energy estimates, $X^{(1)}_5$ only contains terms of degree $5$ and $X^{(1)}_{\geq7}$ contains the remaining terms of degree $\geq7$. In this way, the vector field $X^{(1)}$ in \eqref{X1} can be written as 
	\begin{equation*}\label{alte}
		X^{(1)}(w,\w)=\left(\Id+\P(w,\w)\right)(\D_1+X_{{\bf Res},3}(w,\w))+X^{(1)}_5+X^{(1)}_{\geq7}.
	\end{equation*}
	
	\begin{itemize}
		\item[{\bf 1.}] Since the operator $K_2$ is quadratic, the degree-5 part of 
		$$K_2(w,\w)\left(\Id+K_2(w,\w)\right)^{-1}\left(\B_3(w,\w)-X_{{\bf Res}, 3}(w,\w)\right)$$
		is
		\begin{equation}\label{51}
			K_2(w,\w)\left(\B_3(w,\w)-X_{{\bf Res}, 3}(w,\w)\right).
		\end{equation}
		\item[{\bf2.}] By recalling the definition of $\P$ in \eqref{P}, we have that the degree-5 part of
		$$-\P(w,\w)\left(\Id+K_2(w,\w)\right)^{-1}\left(\B_3(w,\w)-X_{{\bf Res}, 3}(w,\w)\right)$$
		is
		\begin{equation}\label{52}
			cQ(w,\w)\left(\B_3(w,\w)-X_{{\bf Res}, 3}(w,\w)\right).
		\end{equation}
		\item[{\bf 3.}] By recalling the definition of $\B_5$ in \eqref{R5} we have that the degree-5 part of $$\left(\Id+K_2(w,\w)\right)^{-1}\B_{5}\left(\Phi^{(3)}(w,\w)\right)$$
		is
		\begin{equation}\label{53}
			-cQ(w,\w)\B_3(w,\w).
		\end{equation}
		\item[{\bf4.}] Finally, concerning the term
		$$\left[\B_3\left(\Phi^{(3)}(w,\w)\right)-\B_3(w,\w)\right],$$
		by recalling the definition of $\Phi^{(3)}=\Id+M$ and Taylor expanding in $(w,\w)$ we obtain
		\begin{equation}\label{54}
			\B'_3(w,\w)M(w,\w)
			\begin{pmatrix}
				w\\
				\w
			\end{pmatrix},
		\end{equation}
		where $\B'_3(u,v)$ is the Gateaux derivative of $\B_3$ in the point  $(u,v)$ and acts on pair of functions $(\alpha,\beta)$ in the following way:
		\begin{eqnarray*}\label{B3'}
			\B'_3(u,v)
			\begin{pmatrix}
				\alpha\\
				\beta
			\end{pmatrix}
			&=&-ic\big(\left\langle\Lambda v,\Lambda \beta\right\rangle-\left\langle \Lambda u,\Lambda \alpha\right\rangle\big)
			\begin{pmatrix}
				v\\
				u
			\end{pmatrix}\nonumber\\\\
			&&-\frac{ic}{2}\big(\left\langle\Lambda v,\Lambda v\right\rangle-\left\langle \Lambda u,\Lambda u\right\rangle\big)
			\begin{pmatrix}
				\beta\\
				\alpha
			\end{pmatrix}.\nonumber
		\end{eqnarray*}
	\end{itemize}
	By collecting \eqref{51}, \eqref{52}, \eqref{53} and \eqref{54} we have
	\begin{equation}\label{X51}
		X^{(1)}_5(w,\w)=-K_2(w,\w)X_{{\bf Res}, 3}(w,\w)+\B'_3(w,\w)M(w,\w)
		\begin{pmatrix}
			w\\
			\w
		\end{pmatrix}
	\end{equation}
	and, consequently
	\begin{eqnarray*}\label{Xgeq71}
		X^{(1)}_{\geq7}(w,\w)&=&K_2(w,\w)\left(-K_2(w,\w)+\tilde{K}_2(w,\w)\right)\left(\B_3(w,\w)-X_{{\bf Res},3}(w,\w)\right)\nonumber\\
		&&-\P(w,\w)\left(-K_2(w,\w)+\tilde{K}_2(w,\w)\right)\left(\B_3(w,\w)-X_{{\bf Res},3}(w,\w)\right)\nonumber\\
		&&+\P(w,\w)\left[\B_3\left(\Phi^{(3)}(w,\w)\right)-\B_3(w,\w)\right]\\
		&&+\left(-K_2(w,\w)+\tilde{K}_2(w,\w)\right)\B_{5}\left(\Phi^{(3)}(w,\w)\right)\nonumber\\
		&&+\left[\B_3\left(\Phi^{(3)}(w,\w)\right)-\B_3(w,\w)-\B'_3(w,\w)M(w,\w)
		\begin{pmatrix}
			w\\
			\w
		\end{pmatrix}\right]\nonumber\\
		&&-K_2(w,\w)\left[\B_3\left(\Phi^{(3)}(w,\w\right)-\B_3(w,\w)\right]+\tilde{K}_2(w,\w)\B_3\left(\Phi^{(3)}(w,\w)\right)\nonumber.
	\end{eqnarray*}
	From a direct calculation we have:
	\begin{lemma}
		For all $s\geq0$, all $(w,z)\in  H^s_0(\T^n,c.c.)\cap H^{m_0}_0(\T^n,c.c.)$ with $\|w\|_{m_0}\leq\delta_0$, there exists a positive constant $\mathfrak{C}$ such that the following hold:
		\begin{equation*}\label{ema0}\left\|X^{(1)}_{5}(w,\w)\right\|_s\leq\mathfrak{C}|c|^2\|w\|^2_{m_0}\|w\|^2_1\|w\|_s,\quad\left\|X^{(1)}_{\geq7}(w,\w)\right\|_s\leq \mathfrak{C}|c|^3. \|w\|^4_{m_0}\|w\|^2_1\|w\|_s
		\end{equation*}
	\end{lemma}
	
	Let us consider the following change of variables:
	\begin{equation}\label{Phi4}
		\begin{pmatrix}
			w\\
			\w
		\end{pmatrix}=\Phi^{(4)}(q,\qb):=
		\left(\Id+\MM(q,\qb)\right)
		\begin{pmatrix}
			q\\
			\qb
		\end{pmatrix},
	\end{equation}
	where
	\begin{equation}\label{bhhs}
		\MM(q,\qb)=\A[q,q,q,q]+\B[q,q,q,\qb]+\Cc[q,q,\qb,\qb]+\D[q,\qb,\qb,\qb]+\Ff[\qb,\qb,\qb.\qb]
	\end{equation}
	is a generic map in the space of quadrilinear operators.
	We write 
	$$\A[u,u,u,u]=
	\begin{pmatrix}
		\A_{11}[q,q,q,q]&\A_{12}[q,q,q,q]\\
		\A_{21}[q,q,q,q]&\A_{22}[q,q,q,q]
	\end{pmatrix},
	$$
	with $$\A_{11}[q^{(1)},q^{(2)},q^{(3)},q^{(4)}]h=\sum_{j,\ell,k}q^{(1)}_jq^{(2)}_{-j}q^{(3)}_\ell q^{(4)}_{-\ell}\,h_k\,a_{11}(j,\ell,k)e^{ik\cdot x}$$
	and similarly for the other terms.\\
	We assume moreover the following symmetry relations:
	\begin{align*}
		\A[q^{(1)},q^{(2)},q^{(3)},q^{(4)}]&=\A[q^{(2)},q^{(1)},q^{(3)},q^{(4)}]=\A[q^{(1)},q^{(2)},q^{(4)},q^{(3)}]\\
		\B[q^{(1)},q^{(2)},q^{(3)},\qb]&=\B[q^{(2)},q^{(1)},q^{(3)},\qb]\\
		\Cc[q^{(1)},q^{(2)},\qb^{(1)},\qb^{(2)}]&=\Cc[q^{(2)},q^{(1)},\qb^{(1)},\qb^{(2)}]=\Cc[q^{(1)},q^{(2)},\qb^{(2)},\qb^{(1)}]\\
		\D[q,\qb^{(1)},\qb^{(2)},\qb^{(3)}]&=\D[q,\qb^{(1)},\qb^{(3)},\qb^{(2)}]\\
		\Ff[\qb^{(1)},\qb^{(2)},\qb^{(3)},\qb^{(4)}]&=\Ff[\qb^{(2)},\qb^{(1)},\qb^{(3)},\qb^{(4)}]=\Ff[\qb^{(1)},\qb^{(2)},\qb^{(4)},\qb^{(3)}]
	\end{align*}
	Finally, in order for $\Phi^{(4)}$ to map complex-conjugate pair of functions into pairs of complex-conjugate function, we assume
	$$\MM_{11}(q,\qb)=\overline{\MM_{22}(q,\qb)},\quad\MM_{12}(q,\qb)=\overline{\MM_{21}(q,\qb)}.$$
	
	Similarly to the first normal form transformation, we have in the new set of variables:
	\begin{equation*}
		\left(\Id+\K(q,\qb)\right)\dt
		\begin{pmatrix}
			q\\
			\qb
		\end{pmatrix}=X^{(1)}\left(\Phi^{(4)}(q,\qb)\right)
	\end{equation*}
	where
	\begin{equation*}\label{K3}
		\K(q,\qb)
		\begin{pmatrix}
			\alpha\\
			\beta
		\end{pmatrix}=\left(\MM(q,\qb)+\E(q,\qb)\right)
		\begin{pmatrix}
			\alpha\\
			\beta
		\end{pmatrix}
	\end{equation*}
	and 
	\begin{eqnarray*}\label{E}
		\E(q,\qb)
		\begin{pmatrix}
			\alpha\\
			\beta
		\end{pmatrix}&=&\Big(2\A[q,\alpha,q,q]+2\A[q,q,q,\alpha]+2\B[q,\alpha,q,\qb]+\B[q,q,\alpha,\qb]\nonumber\\
		&&+\B[q,q,q,\beta]+2\Cc[q,\alpha,\qb,\qb]+2\Cc[q,q,\qb,\beta]+\D[\alpha,\qb,\qb,\qb]\\
		&&+\D[q,\beta,\qb,\qb]+2\D[q,\qb,\qb,\beta]+2\Ff[\qb,\beta,\qb,\qb]+2\Ff[\qb,\qb,\qb,\beta]\Big)
		\begin{pmatrix}
			q\\
			\qb
		\end{pmatrix}.\nonumber
	\end{eqnarray*}
	
	Hence, we can write
	\begin{eqnarray*}\nonumber
		\dt
		\begin{pmatrix}
			q\\
			\qb
		\end{pmatrix}&=&X^{(2)}(q,\qb):=\left(\Id+\K(q,\qb)\right)^{-1}X^{(1)}\left(\Phi^{(4)}(q,\qb)\right)\\
		&=&\Bigg(1+\P\left(\Phi^{(4)}(q,\qb)\right)\Bigg)\left(\D_1(q,\qb)+X_{{\bf Res},3}(q,\qb)\right)+X^{(2)}_5(q,\qb)+X^{(2)}_{\geq7}(q,\qb),\nonumber
	\end{eqnarray*}
	where
	\begin{equation*}
		X^{(2)}_5(q,\qb)=X_5^{(1)}(q,\qb)+\D_1\left(\M(q,\qb)
		\begin{pmatrix}
			q\\
			\qb
		\end{pmatrix}\right)-\K(q,\qb)\D_1(q,\qb).
	\end{equation*}
	is the part with degree-5 and 
	\begin{eqnarray}\label{X2geq7}
		X^{(2)}_{\geq7}(q,\qb)&=&\left[1+\P\left(\Phi^{(4)}(q,\qb)\right)\right]\left(-\K(q,\qb)+\tilde{\K}(q,\qb)\right)\left(X^{(2)}_5(q,\qb)-X^{(1)}_5(q,\qb)\right)\nonumber\\
		&&+\P\left(\Phi^{(4)}(q,\qb)\right)\left(X^{(2)}_5(q,\qb)-X^{(1)}_5(q,\qb)\right)\nonumber\\
		&&+\left[1+\P\left(\Phi^{(4)}(q,\qb)\right)\right]\left(-\K(q,\qb)+\tilde{\K}(q,\qb)\right)X_{{\bf Res},3}(q,\qb)\nonumber\\
		&&+\left(-\K(q,\qb)+\tilde{\K}(q,\qb)\right)X_5^{(1)}(q,\qb)\\
		&&+\left[1+\P\left(\Phi^{(4)}(q,\qb)\right)\right]\left(\Id+\K(q,\qb)\right)^{-1}\left[X_{{\bf Res },3}\left(\Phi^{(4)}(q,\qb)\right)-X_{{\bf res},3}(q,\qb)\right]\nonumber\\
		&&+\left(\Id+\K(q,\qb)\right)^{-1}\left[X_5^{(1)}\left(\Phi^{(4)}(q,\qb)\right)-X_5^{(1)}(q,\qb)\right]\nonumber\\
		&&+\left(\Id+\K(q,\qb)\right)^{-1}X_{\geq7}^{(1)}\left(\Phi^{(4)}(q,\qb)\right)\nonumber
	\end{eqnarray}
	is the remaining part with degree $\geq7$.\\
	We start by analyzing the first component of $X^{(2)}_5$:
	\begin{equation}\label{X521}
		\left(X^{(2)}_5(q,\qb)\right)_1=\left(X_5^{(1)}(q,\qb)\right)_1-2i\M_{1,2}(q,\qb)\Lambda \qb-\left(\E(q,\qb)
		\begin{pmatrix}
			-i\Lambda q\\
			i\Lambda \qb
		\end{pmatrix}\right)_1.
	\end{equation}
	By recalling the expressions of $X^{(1)}_5$ in \eqref{X51}and of $K_2$ in \eqref{K2} and recalling that
	$$Q(q,\qb)=\frac{c}{2}\left\langle \Lambda q,q\right\rangle+c\left\langle\Lambda q, \qb\right\rangle+\frac{c}{2}\left\langle \Lambda \qb,\qb\right\rangle$$
	and that
	$$\B_3(q,\qb)=\left(-\frac{ic}{2}\left\langle \Lambda \qb,\Lambda \qb\right\rangle+\frac{ic}{2}\left\langle \Lambda q,\Lambda q\right\rangle\right)
	\begin{pmatrix}
		\qb\\
		q
	\end{pmatrix},$$
	we have
	\begin{eqnarray*}
		\big(-K_2(q,\qb)\xt(q,\qb)\big)_1&=&-2A_{1,2}\left[q,\left(\xt(q,\qb)\right)_1\right]\qb-2C_{1,2}\left[\qb,\left(\xt(q,\qb)\right)_2\right]\qb\\
		&&-A_{1,2}[q,q]\left(\xt(q,\qb)\right)_2-C_{1,2}[\qb,\qb]\left(\xt(q,\qb)\right)_2;
	\end{eqnarray*}
	\begin{eqnarray*}
		&&\left(\B'_3(q,\qb)M(q,\qb)
		\begin{pmatrix}
			q\\
			\qb
		\end{pmatrix}\right)_1\\
		&=&-ic\Bigg(\left\langle\Lambda \qb,\Lambda A_{1,2}[\qb,\qb]q+\Lambda C_{1,2}[q,q]q\right\rangle-\left\langle \Lambda q,\Lambda A_{1,2}[q,q]\qb+\Lambda C_{1,2}[\qb,\qb]\qb\right\rangle\Bigg)\qb\\
		&&-\frac{ic}{2}(\left\langle\Lambda \qb,\Lambda \qb\right\rangle-\left\langle \Lambda q,\Lambda q\right\rangle)\left(A_{1,2}[\qb,\qb]q+C_{1,2}[q,q]q\right)\\\\
		&=&-ic\left\langle\Lambda \qb,\Lambda A_{1,2}[\qb,\qb]q\right\rangle \qb-ic\left\langle\Lambda \qb,\Lambda C_{1,2}[q,q]q\right\rangle \qb
		+ic\left\langle \Lambda q,\Lambda A_{1,2}[q,q]\qb\right\rangle \qb\\
		&&+ic\left\langle \Lambda q,\Lambda C_{1,2}[\qb,\qb]\qb\right\rangle \qb-\frac{ic}{2}\left\langle\Lambda \qb,\Lambda \qb\right\rangle A_{1,2}[\qb,\qb]q-\frac{ic}{2}\left\langle\Lambda \qb,\Lambda \qb\right\rangle C_{1,2}[q,q]q\\
		&&+\frac{ic}{2}\left\langle \Lambda q,\Lambda q\right\rangle A_{1,2}[\qb,\qb]q+\frac{ic}{2}\left\langle \Lambda q,\Lambda q\right\rangle C_{1,2}[q,q]q;
	\end{eqnarray*}
	\begin{eqnarray*}
		-2i\M_{1,2}(q,\qb)\Lambda \qb
		&=&-2i\A_{1,2}[q,q,q,q]\Lambda \qb-2i\B_{1,2}[q,q,q,\qb]\Lambda \qb-2i\Cc_{1,2}[q,q,\qb,\qb]\Lambda \qb\\
		&&-2i\D_{1,2}[q,\qb,\qb,\qb]\Lambda \qb-2i\Ff_{1,2}[\qb,\qb,\qb,\qb]\Lambda \qb;
	\end{eqnarray*}
	\begin{eqnarray*}
		&&-\left(\E(q,\qb)
		\begin{pmatrix}
			-i\Lambda q\\
			i\Lambda \qb
		\end{pmatrix}\right)_1\\
		&=&2i\A_{1,1}[q,\Lambda q,q,q]q+2i\A_{1,1}[q,q,q,\Lambda q]q+2i\B_{1,1}[q,\Lambda q,q,\qb]q+i\B_{1,1}[q,q,\Lambda q,\qb]q\\
		&&-i\B_{1,1}[q,q,q,\Lambda \qb]q+2i\Cc_{1,1}[q,\Lambda q,\qb,\qb]q-2i\Cc_{1,1}[q,q,\qb,\Lambda \qb]q+i\D_{1,1}[\Lambda q,\qb,\qb,\qb]q\\
		&&-i\D_{1,1}[q,\Lambda \qb,\qb,\qb]q-2i\D_{1,1}[q,\qb,\qb,\Lambda \qb]q-2i\Ff_{1,1}[\qb,\Lambda \qb,\qb,\qb]q-2i\Ff_{1,1}[\qb,\qb,\qb,\Lambda \qb]q\\
		&&+2i\A_{1,2}[q,\Lambda q,q,q]\qb+2i\A_{1,2}[q,q,q,\Lambda q]\qb+2i\B_{1,2}[q,\Lambda q,q,\qb]\qb+i\B_{1,2}[q,q,\Lambda q,\qb]\qb\\
		&&-i\B_{1,2}[q,q,q,\Lambda \qb]\qb+2i\Cc_{1,2}[q,\Lambda q,\qb,\qb]\qb-2i\Cc_{1,2}[q,q,\qb,\Lambda \qb]\qb+i\D_{1,2}[\Lambda q,\qb,\qb,\qb]\qb\\
		&&-i\D_{1,2}[q,\Lambda \qb,\qb,\qb]\qb-2i\D_{1,2}[q,\qb,\qb,\Lambda \qb]\qb-2i\Ff_{1,2}[\qb,\Lambda \qb,\qb,\qb]\qb-2i\Ff_{1,2}[\qb,\qb,\qb,\Lambda \qb]\qb;
	\end{eqnarray*}
	Collecting all these terms, we obtain
	\begin{eqnarray*}
		\left(X^{(2)}_5(q,\qb)\right)_1&=&-2A_{1,2}\left[q,\left(X^+_3(q,\qb)\right)_1\right]\qb-2C_{1,2}\left[\qb,\left(X^+_3(q,\qb)\right)_2\right]\qb\\\\
		&&-A_{1,2}[q,q]\left(X^+_3(q,\qb)\right)_2-C_{1,2}[\qb,\qb]\left(X^+_3(q,\qb)\right)_2\\\\
		&&-ic\left\langle\Lambda \qb,\Lambda A_{1,2}[\qb,\qb]q\right\rangle \qb-ic\left\langle\Lambda \qb,\Lambda C_{1,2}[q,q]q\right\rangle \qb
		+ic\left\langle \Lambda q,\Lambda A_{1,2}[q,q]\qb\right\rangle \qb\\\\
		&&+ic\left\langle \Lambda q,\Lambda C_{1,2}[\qb,\qb]\qb\right\rangle \qb-\frac{ic}{2}\left\langle\Lambda \qb,\Lambda \qb\right\rangle A_{1,2}[\qb,\qb]q-\frac{ic}{2}\left\langle\Lambda \qb,\Lambda \qb\right\rangle C_{1,2}[q,q]q\\\\
		&&+\frac{ic}{2}\left\langle \Lambda q,\Lambda q\right\rangle A_{1,2}[\qb,\qb]q+\frac{ic}{2}\left\langle \Lambda q,\Lambda q\right\rangle C_{1,2}[q,q]q\\\\
		&&-2i\A_{1,2}[q,q,q,q]\Lambda \qb-2i\B_{1,2}[q,q,q,\qb]\Lambda \qb-2i\Cc_{1,2}[q,q,\qb,\qb]\Lambda \qb\\\\
		&&-2i\D_{1,2}[q,\qb,\qb,\qb]\Lambda \qb-2i\Ff_{1,2}[\qb,\qb,\qb,\qb]\Lambda \qb\\\\
		&&+2i\A_{1,1}[q,\Lambda q,q,q]q+2i\A_{1,1}[q,q,q,\Lambda q]q+2i\B_{1,1}[q,\Lambda q,q,\qb]q+i\B_{1,1}[q,q,\Lambda q,\qb]q\\\\
		&&-i\B_{1,1}[q,q,q,\Lambda \qb]q+2i\Cc_{1,1}[q,\Lambda q,\qb,\qb]q-2i\Cc_{1,1}[q,q,\qb,\Lambda \qb]q+i\D_{1,1}[\Lambda q,\qb,\qb,\qb]q\\\\
		&&-i\D_{1,1}[q,\Lambda \qb,\qb,\qb]q-2i\D_{1,1}[q,\qb,\qb,\Lambda \qb]q-2i\Ff_{1,1}[\qb,\Lambda \qb,\qb,\qb]q-2i\Ff_{1,1}[\qb,\qb,\qb,\Lambda \qb]q\\\\
		&&+2i\A_{1,2}[q,\Lambda q,q,q]\qb+2i\A_{1,2}[q,q,q,\Lambda q]\qb+2i\B_{1,2}[q,\Lambda q,q,\qb]\qb+i\B_{1,2}[q,q,\Lambda q,\qb]\qb\\\\
		&&-i\B_{1,2}[q,q,q,\Lambda \qb]\qb+2i\Cc_{1,2}[q,\Lambda q,\qb,\qb]\qb-2i\Cc_{1,2}[q,q,\qb,\Lambda \qb]\qb+i\D_{1,2}[\Lambda q,\qb,\qb,\qb]\qb\\\\
		&&-i\D_{1,2}[q,\Lambda \qb,\qb,\qb]\qb-2i\D_{1,2}[q,\qb,\qb,\Lambda \qb]\qb-2i\Ff_{1,2}[\qb,\Lambda \qb,\qb,\qb]\qb-2i\Ff_{1,2}[\qb,\qb,\qb,\Lambda \qb]\qb.
	\end{eqnarray*}
	\bigskip
	\\
	We now select the operators $\A,\,\B,\Cc,\D,\Ff$ in \eqref{bhhs} and the corresponding Fourier symbols $a,\,b,\,c,\,d,\,f$ in order to cancel out as many terms as possible from $X^{(2)}_5$.\\
	Since, in Fourier variables, the terms involved in the expression of $X^{(2)}_5$ are monomials in the variables $q_j,\,\qb_j$ we start by separating the terms with the same homogeneity in the variables $q$ and $\qb$.\\
	For what follows, we will stop indicating explicitly the restrictions on
	the indices in summations and adopt instead the convention $0/0 = 0$ in the coefficients.\\
	{\bf (1) Terms containing the monomial $\mathbf{q_lq_{-l}q_jq_{-j}q_k}$}:\\
	\begin{eqnarray*}
		&&\frac{ic}{2}\left\langle \Lambda q,\Lambda q\right\rangle C_{1,2}[q,q]q+2i\A_{1,1}[q,\Lambda q,q,q]q+2i\A_{1,1}[q,q,q,\Lambda q]q\\
		&=&-\frac{ic^2}{8}\sum_{j,l,k\neq0}\frac{|l|^2|j|^2}{|l|+|k|}q_jq_{-j}q_lq_{-l}q_ke^{ik\cdot x}+2i\sum_{j,l,k\neq0} |j|a_{11}(j,l,k)q_jq_{-j}q_lq_{-l}q_ke^{ik\cdot x}\\
		&&+\sum|l|a_{11}(j,l,k)q_jq_{-j}q_lq_{-l}q_ke^{ik\cdot x}\\
		&=&\sum_{j,l,k\neq0}\left(2ia_{11}(j,l,k)|j|+2ia_{11}(j,l,k)|l|-\frac{ic^2}{8}\frac{|l|^2|j|^2}{|l|+|k|}\right)q_lq_{-l}q_jq_{-j}q_ke^{ik\cdot x}\\
		&=&\sum_{j,l,k\neq0}\left(2ia_{11}(j,l,k)(|j|+|l|)-\frac{ic^2}{16}\frac{|l|^2|j|^2}{|j|+|k|}-\frac{ic^2}{16}\frac{|l|^2|j|^2}{|l|+|k|}\right)q_lq_{-l}q_jq_{-j}q_ke^{ik\cdot x},
	\end{eqnarray*}
	then we can choose $a_{11}$ to be
	\begin{equation}\label{a11}
		a_{11}(j,l,k)=\frac{c^2}{32}\frac{|l|^2|j|^2}{|j|+|l|}\left(\frac{1}{|j|+|k|}+\frac{1}{|l|+|k|}\right).
	\end{equation}
	There are no resonant terms of this form.\\\\
	{\bf (2) Terms containing the monomial $\mathbf{q_jq_{-j}q_l\qb_{l}\qb_{-k}}$:}\\
	\begin{eqnarray*}
		&&-2A_{1,2}\left[q,\left(X^+_3(q,\qb)\right)_1\right]\qb+ic\left\langle \Lambda q,\Lambda A_{1,2}[q,q]\qb\right\rangle \qb-ic\left\langle\Lambda \qb,\Lambda C_{1,2}[q,q]q\right\rangle \qb\\
		&&-2i\B_{1,2}[q,q,q,\qb]\Lambda \qb+2i\B_{1,2}[q,\Lambda q,q,\qb]\qb+i\B_{1,2}[q,q,\Lambda q,\qb]\qb-i\B_{1,2}[q,q,q,\Lambda \qb]\qb\\\\
		&=&\frac{ic^2}{4}\sum_{j,l,k\neq0}\frac{|j|^2|l|^2(1-\delta_{|l|}^{|k|})\delta_{|l|}^{|j|}}{|l|-|k|}q_jq_{-j}q_l\qb_{l}\qb_{-k}e^{ik\cdot x}-\frac{ic^2}{4}\sum_{j,l,k\neq0}\frac{|j|^2|l|^2(1-\delta_{|l|}^{|j|})}{|j|-|l|}q_jq_{-j}q_l\qb_{l}\qb_{-k}e^{ik\cdot x}\\
		&&+\frac{ic^2}{4}\sum_{j,l,k\neq0}\frac{|j|^2|l|^2}{|j|+|l|}q_jq_{-j}q_l\qb_{l}\qb_{-k}e^{ik\cdot x}-2i\sum_{j,l,k\neq0}|k|b_{12}(j,l,k)q_jq_{-j}q_l\qb_{l}\qb_{-k}e^{ik\cdot x}\\
		&&+2i\sum_{j,l,k\neq0}|j|b_{12}(j,l,k)q_jq_{-j}q_l\qb_{l}\qb_{-k}e^{ik\cdot x}+i\sum_{j,l,k\neq0}|l|b_{12}(j,l,k)q_jq_{-j}q_l\qb_{l}\qb_{-k}e^{ik\cdot x}\\
		&&-i\sum_{j,l,k\neq0}|l|b_{12}(j,l,k)q_jq_{-j}q_l\qb_{l}\qb_{-k}e^{ik\cdot x}\\
		&=&\sum_{j,l,k\neq0}\Bigg(\frac{ic^2}{4}\frac{|j|^2|l|^2(1-\delta_{|l|}^{|k|})\delta_{|l|}^{|j|}}{|l|-|k|}-\frac{ic^2}{4}\frac{|j|^2|l|^2(1-\delta_{|l|}^{|j|})}{|j|-|l|}+\frac{ic^2}{4}\frac{|j|^2|l|^2}{|j|+|l|}\\
		&&+2ib_{12}(j,l,k)(|j|-|k|)\Bigg)q_jq_{-j}q_l\qb_{l}\qb_{-k}e^{ik\cdot x},
	\end{eqnarray*}
	hence
	\begin{equation}\label{b12}
		b_{1,2}(j,l,k)=\frac{c^2|j|^2|l|^2}{8}\left(\frac{(1-\delta_{|l|}^{|k|})\delta_{|l|}^{|j|}}{|l|-|k|}-\frac{(1-\delta_{|l|}^{|j|})}{|j|-|l|}+\frac{1}{|j|+|l|}\right)\frac{1-\delta_{|j|}^{|k|}}{|k|-|j|}.
	\end{equation}
	The remaining resonant terms are given by
	$$\sum_{|j|=|k|}\frac{ic^2|j|^2|l|^2}{8}\Bigg(\frac{(1-\delta_{|l|}^{|k|})\delta_{|l|}^{|j|}}{|l|-|k|}-\frac{1-\delta_{|l|}^{|j|}}{|j|-|l|}+\frac{1}{|j|+|l|})\Bigg)q_jq_{-j}q_l\qb_{l}\qb_{-k}e^{ik\cdot x}$$
	\begin{equation}\label{R1}\tag{R1}
		=\sum_{|j|=|k|}\frac{ic^2|j|^2|l|^2}{8}\Bigg(\frac{1}{|j|+|l|}-\frac{1-\delta_{|l|}^{|j|}}{|j|-|l|}\Bigg)q_jq_{-j}q_l\qb_{l}\qb_{-k}e^{ik\cdot x}.
	\end{equation}
	{\bf (3) Terms containing the monomial $\mathbf{q_j\qb_{j}\qb_{-l}\qb_{l}\qb_{-k}}$}:\\
	\begin{eqnarray*}
		&&-2C_{1,2}\left[\qb,\left(X^+_3(q,\qb)\right)_2\right]\qb-ic\left\langle\Lambda \qb,\Lambda A_{1,2}[\qb,\qb]q\right\rangle \qb+ic\left\langle \Lambda q,\Lambda C_{1,2}[\qb,\qb]\qb\right\rangle \qb\\
		&&-2i\D_{1,2}[q,\qb,\qb,\qb]\Lambda \qb+i\D_{1,2}[\Lambda q,\qb,\qb,\qb]\qb-i\D_{1,2}[q,\Lambda \qb,\qb,\qb]\qb-2i\D_{1,2}[q,\qb,\qb,\Lambda \qb]\qb\\\\
		&=&-\frac{ic^2}{4}\sum_{j,l,k\neq0}\frac{|j|^2|l|^2\delta^{|j|}_{|l|}}{|j|+|k|}q_j\qb_{j}\qb_{-l}\qb_{l}\qb_{-k}e^{ik\cdot x}+\frac{ic^2}{4}\sum_{j,l,k\neq0}\frac{|j|^2|l|^2(1-\delta^{|j|}_{|l|})}{|l|-|j|}q_j\qb_{j}\qb_{-l}\qb_{l}\qb_{-k}e^{ik\cdot x}\\
		&&-\frac{ic^2}{4}\sum_{j,l,k\neq0}\frac{|j|^2|l|^2}{|l|+|j|}q_j\qb_{j}\qb_{-l}\qb_{l}\qb_{-k}e^{ik\cdot x}-2i\sum_{j,l,k\neq0}d_{12}(j,l,k)(|j|+|k|)q_j\qb_{j}\qb_{-l}\qb_{l}\qb_{-k}e^{ik\cdot x}\\
		&=&\sum_{j,l,k\neq0}\Bigg(-\frac{ic^2}{4}\frac{|j|^2|l|^2\delta^{|j|}_{|l|}}{|j|+|k|}+\frac{ic^2}{4}\frac{|j|^2|l|^2(1-\delta^{|j|}_{|l|})}{|l|-|j|}-\frac{ic^2}{4}\frac{|j|^2|l|^2}{|l|+|j|}\\
		&&-2id_{12}(j,l,k)(|j|+|k|)\Bigg) q_j\qb_{j}\qb_{-l}\qb_{l}\qb_{-k}e^{ik\cdot x}
	\end{eqnarray*}
	hence
	\begin{equation}\label{d12}
		d_{12}(j,l,k)=\frac{ic^2|j||l|}{8}\left(-\frac{\delta^{|j|}_{|l|}}{|j|+|k|}+\frac{1-\delta^{|j|}_{|l|}}{|l|-|j|}-\frac{1}{|l|+|j|}\right)\frac{1}{|j|+|k|}.
	\end{equation}
	There are no resonant terms of this form.\\\\
	{\bf (4)  Terms containing the monomial $\mathbf{\qb_{-j}\qb_{j}\qb_{-l}\qb_{l}q_k}$:}
	\begin{eqnarray*}
		&&-C_{1,2}[\qb,\qb]\left(X^+_3(q,\qb)\right)_2-\frac{ic}{2}\left\langle\Lambda \qb,\Lambda \qb\right\rangle A_{1,2}[\qb,\qb]q-2i\Ff_{1,1}[\qb,\Lambda \qb,\qb,\qb]q-2i\Ff_{1,1}[\qb,\qb,\qb,\Lambda \qb]q\\
		&=&-\frac{ic^2}{8}\sum_{j,l,k\neq0}\frac{|j|^2|l|^2\delta^{|k|}_{|l|}}{|j|+|k|}\qb_{-j}\qb_{j}\qb_{-l}\qb_{l}q_ke^{ik\cdot x}+\frac{ic^2}{8}\sum_{j,l,k\neq0}\frac{|j|^2|l|^2(1-\delta^{|k|}_{|l|})}{|j|-|k|}\qb_{-j}\qb_{j}\qb_{-l}\qb_{l}q_ke^{ik\cdot x}\\
		&&-2i\sum_{j,l,k\neq0}f_{11}(j,l,k)(|j|+|l|)\qb_{-j}\qb_{j}\qb_{-l}\qb_{l}q_ke^{ik\cdot x}\\
		&=&\sum_{j,l,k\neq0}\left(-\frac{ic^2}{8}\frac{|j|^2|l|^2\delta^{|k|}_{|l|}}{|j|+|k|}+\frac{ic^2}{8}\frac{|j|^2|l|^2(1-\delta^{|k|}_{|l|})}{|j|-|k|}  -2if_{11}(j,l,k)(|j|+|l|)\right)\qb_{-j}\qb_{j}\qb_{-l}\qb_{l}q_ke^{ik\cdot x}
	\end{eqnarray*}
	hence
	\begin{equation}\label{f11}
		f_{11}(j,l,k)=\frac{c^2|j|^2|l|^2}{16}\left(-\frac{\delta^{|k|}_{|l|}}{|j|+|k|}+\frac{(1-\delta^{|k|}_{|l|})}{|j|-|k|}\right)\frac{1}{|j|+|l|}
	\end{equation}
	There are no remaining resonant terms of this form.\\\\
	{\bf (5) Terms containing the monomial $\mathbf{q_jq_{-j}\qb_{-l}\qb_{l}q_k}$:}
	\begin{eqnarray*}
		&&-A_{1,2}[q,q]\left(X^+_3(q,\qb)\right)_2-\frac{ic}{2}\left\langle\Lambda \qb,\Lambda \qb\right\rangle C_{1,2}[q,q]q+\frac{ic}{2}\left\langle \Lambda q,\Lambda q\right\rangle A_{1,2}[\qb,\qb]q\\
		&&+2i\Cc_{1,1}[q,\Lambda q,\qb,\qb]q-2i\Cc_{1,1}[q,q,\qb,\Lambda \qb]q\\\\
		&=&-\frac{ic^2}{8}\sum_{j,l,k\neq0}\frac{|j|^2|l|^2(1-\delta_{|j|}^{|k|})\delta_{|l|}^{|k|}}{|j|-|k|}q_jq_{-j}\qb_{-l}\qb_{l}q_ke^{ik\cdot x}+\frac{ic^2}{8}\sum_{j,l,k\neq0}\frac{|j|^2|l|^2}{|j|+|k|}q_jq_{-j}\qb_{-l}\qb_{l}q_ke^{ik\cdot x}\\
		&&-\frac{ic^2}{8}\sum_{j,l,k\neq0}\frac{|j|^2|l|^2(1-\delta_{|l|}^{|k|})}{|l|-|k|}q_jq_{-j}\qb_{-l}\qb_{l}q_ke^{ik\cdot x}-2i\sum_{j,l,k\neq0}c_{11}(j,l,k)(|l|-|j|)q_jq_{-j}\qb_{-l}\qb_{l}q_ke^{ik\cdot x}\\
		&=&\sum_{j,l,k\neq0}\Big(-\frac{ic^2}{8}\frac{|j|^2|l|^2(1-\delta_{|j|}^{|k|})\delta_{|l|}^{|k|}}{|j|-|k|}+\frac{ic^2}{8}\frac{|j|^2|l|^2}{|j|+|k|}-\frac{ic^2}{8}\frac{|j|^2|l|^2(1-\delta_{|l|}^{|k|})}{|l|-|k|}\\
		&&-2ic_{11}(j,l,k)(|l|-|j|) \Big)q_jq_{-j}\qb_{-l}\qb_{l}q_ke^{ik\cdot x}
	\end{eqnarray*}
	hence
	\begin{equation}\label{c11}
		c_{11}(j,l,k)=\frac{c^2|j|^2|l|^2}{16}\left(-\frac{(1-\delta_{|j|}^{|k|})\delta_{|l|}^{|k|}}{|j|-|k|}+\frac{1}{|j|+|k|}-\frac{(1-\delta_{|l|}^{|k|})}{|l|-|k|}\right)\frac{1-\delta_{|j|}^{|l|}}{|l|-|j|}.
	\end{equation}
	The remaining resonant terms are
	\begin{equation}\label{R2}\tag{R2}
		\sum_{|j|=|l|}\frac{ic^2|j|^2|l|^2}{8}\left(\frac{1}{|j|+|k|}-\frac{1-\delta_{|l|}^{|k|}}{|l|-|k|} \right)q_jq_{-j}\qb_{-l}\qb_{l}q_ke^{ik\cdot x}.
	\end{equation}
	
	As they are not needed in the calculation, we set
	$$\A_{12},\,\B_{11},\,\Cc_{12},\,\D_{11},\,\Ff_{12}=0.$$
	\\\\
	By recalling \eqref{R1} and \eqref{R2}, we have that the resonant vector field of degree $5$ has first component
	\begin{eqnarray}\label{W51}
		\left(X_{{\bf Res}, 5}(q,\qb)\right)_1&=&\sum_{|j|=|k|}\frac{ic^2|j|^2|l|^2}{8}\Bigg(\frac{1}{|j|+|l|}-\frac{1-\delta_{|l|}^{|j|}}{|j|-|l|}\Bigg)q_jq_{-j}q_l\qb_{l}\qb_{-k}e^{ik\cdot x}\nonumber\\\\
		&&+\sum_{|j|=|l|}\frac{ic^2|j|^2|l|^2}{8}\left(\frac{1}{|j|+|k|}-\frac{1-\delta_{|l|}^{|k|}}{|l|-|k|} \right)q_jq_{-j}\qb_{-l}\qb_{l}q_ke^{ik\cdot x}.\nonumber
	\end{eqnarray}
	\\\\
	Analogous calculations lead to
	\begin{eqnarray*}
		&&\A_{21},\,\B_{22},\,\Cc_{21},\,\D_{22},\,\Ff_{21}=0,\\
		&&\A_{22}=\Ff_{11},\,\B_{21}=\Cc_{12},\,\Cc_{22}=\B_{11},\,\Ff_{22}=\A_{11}
	\end{eqnarray*}
	and
	\begin{eqnarray}\label{W52}
		\left(X_{{\bf Res}, 5}(q,\qb)\right)_2=&&-\sum_{|j|=|k|}\frac{ic^2|j|^2|l|^2}{8}\Bigg(\frac{1}{|j|+|l|}-\frac{1-\delta_{|l|}^{|j|}}{|j|-|l|}\Bigg)\qb_{-j}\qb_{j}\qb_{-l}q_{-l}q_ke^{ik\cdot x}\nonumber\\\\
		&&-\sum_{|j|=|l|}\frac{ic^2|j|^2|l|^2}{8}\left(\frac{1}{|j|+|k|}-\frac{1-\delta_{|l|}^{|k|}}{|l|-|k|} \right)\qb_{-j}\qb_{j}q_lq_{-l}\qb_{-k}e^{ik\cdot x}.\nonumber
	\end{eqnarray}
	
	\begin{lemma}
		The vector field $X^{(2)}$ preserves the real structure, namely $\overline{\Big(X^{(2)}\Big)}_1=\Big(X^{(2)}\Big)_2$
	\end{lemma}
	\begin{proof}
		It follows from the definitions of $a_{11},\,b_{12},\,c_{11},\,d_{12},\,f_{11}$ in \eqref{a11}, \eqref{b12}, \eqref{a11}, \eqref{c11}, \eqref{d12}, \eqref{f11}.
	\end{proof}
	We now collect some basic estimates on the operators $\A,\,\B,\,\Cc,\,\D,\,\Ff$:

	\begin{lemma}
		For any complex functions $u,v,w,z,h$ and for all real $s\geq0$, it holds:
		\begin{eqnarray}\label{estimp5}
			\|\A_{11}[u,v,w,z]h\|_s&\leq& c^2\|u\|_{\frac{1}{2}}\|v\|_{\frac{1}{2}}\|w\|_{\frac{1}{2}}\|z\|_{\frac{1}{2}}\|h\|_s\nonumber\\
			\|B_{12}[u,v,w,z]h\|_s&\leq&3c^2\|u\|_{m_0}\|v\|_{m_0}\|w\|_{m_0}\|z\|_{m_0}\|h\|_s\nonumber\\
			\|\Cc_{11}[u,v,w,z]h\|_s&\leq&2c^2\|u\|_{m_0}\|v\|_{m_0}\|w\|_{m_0}\|z\|_{m_0}\|h\|_s\\
			\|\D_{12}[u,v,w,z]h\|_s&\leq&c^2\|u\|_1\|v\|_1\|w\|_1\|z\|_1\|h\|_s\nonumber\\
			\|\Ff_{11}[u,v,w,z]h\|_s&\leq&c^2\|u\|_1\|v\|_1\|w\|_1\|z\|_1\|h\|_s\nonumber
		\end{eqnarray}
	\end{lemma}
	\begin{proof}
		It follows directly from \eqref{a11}, \eqref{b12} , \eqref{d12}, \eqref{f11}, \eqref{c11}, together with a repeated use of \eqref{trick}.
	\end{proof}
	\begin{lemma}
		For all $s\geq0$, all $(u,v)\in H^{m_0}_0(\T^n,c.c.),\,(\alpha,\beta)\in H^s_0(\T^n,c.c.)$ one has
		\begin{eqnarray}
			\label{estimM2}
			\left\|\M(u,v)
			\begin{pmatrix}
				\alpha\\
				\beta
			\end{pmatrix}\right\|_s&\leq& 12c^2\|u\|^4_{m_0}\|\alpha\|_s\\
			\label{estimK2}
			\left\|\K(u,v)
			\begin{pmatrix}
				\alpha\\
				\beta
			\end{pmatrix}\right\|_s&\leq&12c^2\|u\|^3_{m_0}(\|u\|_{m_0}\|\alpha\|_s+4\|u\|_s\|\alpha\|_{m_0}).\\
			\nonumber
		\end{eqnarray}
		Moreover, there exists $\delta_3$ and a positive constant $\mathfrak{C}$ depending on $c$, such that, if $\|u\|_{m_0}<\delta_3$, the operator
		$$\left(\Id+\K(u,v)\right):H^{m_0}_0(\T^n,c.c.)\longrightarrow H^{m_0}_0(\T^n,c.c.)$$
		is invertible with inverse satisfying
		\begin{equation}\label{estiminvK}
			\left\|\left(\Id+\K(u,v)\right)^{-1}
			\begin{pmatrix}
				\alpha\\
				\beta
			\end{pmatrix}\right\|_s\leq \mathfrak{C}\left(\|\alpha\|_s+\|u\|^3_{m_0}\|u\|_s\|\alpha\|_{m_0}\right)
			.\end{equation}
	\end{lemma}
	\begin{proof}
		Estimates \eqref{estimM2} and \eqref{estimK2} are a direct consequence of \eqref{estimp5}, while \eqref{estiminvK} follows from an application of the Neumann series together with \eqref{estimK2} with
		$$\delta_3=\left(60c^2\right)^{-\frac {1}{4}}.$$
	\end{proof}
	\begin{lemma}\label{contract2}
		There exists a universal constant $\delta_4>0$ such that, for all $(w,\w)\in H^{m_0}_0(\T^n,c.c.)$ in the ball $\|w\|_{m_0}\leq\delta$ there exists a unique $(q,\qb)\in H^{m_0}_0(\T^n,c.c.)$ such that $\Phi^{(4)}(q,\qb)=(w,\w)$, with $\|(q,\qb)\|_{m_0}\leq2\|(w,\w)\|_{m_0}$.\\
		If, in addition, $(w,\w)\in H_0^s(\T^n,c.c)$ for some $s\geq m_0$, then $(q,\qb)$ also belongs to $H^s_0(\T^n,c.c)$ and
		\begin{equation}\label{bond4}
		\|(q,\qb)\|_s\leq2\|(w,\w)\|_s.
		\end{equation}\\
		(in our case 
		\begin{equation}\label{delta3}
			\delta_4=\left(385c^2\right)^{-\frac{1}{4}}
		\end{equation}
	\end{lemma}
	\begin{proof}
		The proof is the same as in Lemma \ref{contract1} with the use of \eqref{estimM2} and \eqref{estimK2}.
	\end{proof}
	\vspace{1cm}
	If we perform an energy estimate on a solution of
	$$
	\dt
	\begin{pmatrix}
		q\\
		\qb
	\end{pmatrix}=X^{(2)}(q,\qb)=\Bigg(1+\P\left(\Phi^{(4)}(q,\qb)\right)\Bigg)\left(\D_1(q,\qb)+X_{{\bf Res},3}(q,\qb)\right)+X_{{\bf Res},5}(q,\qb)+X^{(2)}_{\geq7}(q,\qb)
	$$
	since the terms $ \Bigg(1+\P\left(\Phi^{(4)}(q,\qb)\right)\Bigg)\left(\D_1(q,\qb)+X_{{\bf Res},3}(q,\qb)\right)$ give zero contribution, we obtain
	\begin{eqnarray}\label{su3}
		\dt\left(\|q\|^2_s\right)&=&\left\langle\left(X^{(2)}(q,\qb)\right)_1,\Lambda^{2s}\qb\right\rangle+\left\langle\Lambda^{2s} q,\left(X^{(2)}(q,\qb)\right)_2\right\rangle\\
		&=&\left\langle\left(X_{{\bf Res},5}(q,\qb)\right)_1+\left(X^{(2)}_{\geq 7}(q,\qb)\right)_1,\Lambda^{2s}\qb\right\rangle+\left\langle\Lambda^{2s} q,\left(X_{{\bf Res},5}(q,\qb)\right)_2+\left(X^{(2)}_{\geq 7}(q,\qb)\right)_2\right\rangle\nonumber.
	\end{eqnarray}
	Then the first non-trivial contribution is given by $X_{{\bf Res},5}$.\\
	By recalling \eqref{W51} and \eqref{W52} we have
	\begin{eqnarray}
		&&\left\langle\left(X_{{\bf Res},5}(q,\qb)\right)_1,\Lambda^{2s}\qb\right\rangle+\left\langle\Lambda^{2s} q,\left(X_{{\bf Res},5}(q,\qb)\right)_2\right\rangle\nonumber\\
		&=&\sum_{|j|=|k|}\frac{ic^2|j|^2|l|^2}{8}\Bigg(\frac{1}{|j|+|l|}-\frac{1-\delta_{|l|}^{|j|}}{|j|-|l|}\Bigg)q_jq_{-j}q_l\qb_{l}\qb_{-k}\qb_k\label{cI}\\
		&&+\sum_{|j|=|l|}\frac{ic^2|j|^2|l|^2}{8}\left(\frac{1}{|j|+|k|}-\frac{1-\delta_{|l|}^{|k|}}{|l|-|k|}) \right)q_jq_{-j}\qb_{-l}\qb_{l}q_k\qb_k\label{cII}\\
		&&-\sum_{|j|=|k|}\frac{c^2|j|^2|l|^2}{8}\Bigg(\frac{1}{|j|+|l|}-\frac{1-\delta_{|l|}^{|j|}}{|j|-|l|}\Bigg)\qb_{-j}\qb_{j}\qb_{-l}q_{-l}q_kq_{-k}\nonumber\label{cIII}\\
		&&+\sum_{|j|=|l|}\frac{ic^2|j|^2|l|^2}{8}\left(\frac{1}{|j|+|k|}-\frac{1-\delta_{|l|}^{|k|}}{|l|-|k|}) \right)\qb_{-j}\qb_{j}q_lq_{-l}\qb_{-k}q_{-k}\nonumber\label{cIIII}.
	\end{eqnarray}
	By renaming $j\leftrightarrow k$ in equation \eqref{cI} and $j\leftrightarrow l$ in \eqref{cII} we have that
	\begin{equation}\label{cancellazione2}
		\left\langle\left(X_{{\bf Res},5}(q,\qb)\right)_1,\Lambda^{2s}\qb\right\rangle+\left\langle\Lambda^{2s} q,\left(X_{{\bf Res},5}(q,\qb)\right)_2\right\rangle=0.
	\end{equation}
	\\\\
	Equation \eqref{cancellazione2} implies that also $X_{{\bf Res},5}$ does not contribute to the energy estimate, therefore we need to quantify the contribution of the higher order term $X^{(2)}_{\geq7}$, defined in \eqref{X2geq7}.
	
	\begin{lemma}
		For all $s\geq0$ all $(q,\qb)\in H_0^s(\T^n,c.c.)\cap B_{\delta_4}\left(H_0^{m_0}(\T^n,c.c.)\right)$ one has 
		\begin{equation}\label{stimX2geq7}
			\left\|X^{(2)}_{\geq7}(q,\qb)\right\|_s\leq \c_1\|q\|_{m_0}^4\|q\|_{1}^2\|q\|_s
		\end{equation}
		where $\c_1$ is a constant depending on $c$.
	\end{lemma}
	\begin{proof}
		It follows from the definition of $X^{(2)}_{\geq7}$ in \eqref{X2geq7} together with estimates \eqref{estimK2} and \eqref{estimM}.
	\end{proof}
	We can therefore improve the energy estimate on the solutions of the new system:
	\begin{lemma}
		Let $T>0$, $s\geq m_0$ and consider a solution $(q,\qb)\in H_0^s(\T^n,c.c.)\cap B_{\delta_4}\left(H_0^{m_0}(\T^n,c.c.)\right)$ of
		$$\dt(q,\qb)=X^{(2)}(q,\qb),$$
		then it holds
		\begin{equation}\label{ema1}
			\dt(\|q\|^2_s)\leq \c_2\|q\|_{m_0}^4\|q\|_{1}^2\|q\|^2_s
		\end{equation}
		where $\c_2$ is a constant depending on $c$.
	\end{lemma}
	\begin{proof}
		Follows from \eqref{su3}, together with \eqref{cancellazione2} and \eqref{stimX2geq7}.
	\end{proof}
\section{Proof of the main results}\label{proofs}
In this section we collect the results of Sections \ref{Preliminary Transformations}, \ref{first step} and \ref{secondstep}
to prove Theorem \ref{Theorem1} and Theorem \ref{theorem!}.

Throughout, we denote by
\begin{equation}\label{Phi}
\Phi:=\Phi^{(1)}\circ\Phi^{(2)}\circ\Phi^{(3)}\circ\Phi^{(4)}
\end{equation}
the composition of the four transformations constructed in \eqref{change},
\eqref{Phi2}, \eqref{change4} and \eqref{Phi4}, namely the linear diagonalization
of the highest order, the block-diagonalization, and the two normal form steps.\\
It holds the following result:

	\begin{proposition}\label{prop:Phi:global}
		Let us consider the map $\Phi$ defined in \eqref{Phi}. There exist $\delta,\,\mathfrak{C}_0>0$, depending only on $c$, such that, the
		following holds for every $s\geq m_0$:\\\\
		{\bf 1.\ }{\bf Correspondence of data:}\\
			For every pair of zero mean real functions
			$(u,v)\in H_0^{s+\frac12}(\T^n,\R)\times H_0^{s-\frac12}(\T^n,\R)$
			satisfying
			\begin{equation}\label{miserve33}
			\|u\|_{m_0+\frac12}+\|v\|_{m_0-\frac12}\leq\delta,
			\end{equation}
			there exists a unique pair $(q,\qb)\in H_0^s(\T^n,\text{c.c.})$ such that
			$(u,v)=\Phi(q,\qb)$. Moreover it holds
			\begin{equation}\label{miserve22}
			\|q\|_s\leq\mathfrak{C}_0\bigl(\|u\|_{s+\frac12}+\|v\|_{s-\frac12}\bigr).
			\end{equation}
			Conversely, if $(q,\qb)\in H_0^s(\T^n,\text{c.c.})$ satisfies
			$\|q\|_{m_0}\leq\delta$, then $(u,v):=\Phi(q,\qb)$ is a pair of zero mean
			real functions in $H_0^{s+\frac12}(\T^n,\R)\times H_0^{s-\frac12}(\T^n,\R)$
			satisfying
			\begin{equation}\label{miserve22bis}
			\|u\|_{s+\frac12}+\|v\|_{s-\frac12}\leq\mathfrak{C}_0\|q\|_s.
			\end{equation}
			\\
		{\bf 2.\ } {\bf Correspondence of solutions:}\\
			Let $T>0$. If
			\[
			u\in C^0\bigl([0,T],H_0^{s+\frac12}(\T^n,\R)\bigr)\cap
			C^1\bigl([0,T],H_0^{s-\frac12}(\T^n,\R)\bigr)
			\]
			is a solution of \eqref{MK} satisfying
			\begin{equation}\label{bbbb}
			\max_{t\in[0,T]}\Bigl(\|u(t)\|_{m_0+\frac12}+\|\dt u(t)\|_{m_0-\frac12}\Bigr)
			\leq\delta,
			\end{equation}
			then $(q,\qb):=\Phi^{-1}(u,\dt u)\in C^0\bigl([0,T],H_0^s(\T^n,\text{c.c.})\bigr)$
			is a solution of
			\begin{equation}\label{miserve44}
			\dt(q,\qb)=X^{(2)}(q,\qb),
			\end{equation}
			and
			\[
			\max_{t\in[0,T]}\|q(t)\|_s\leq\mathfrak{C}_0
			\max_{t\in[0,T]}\Bigl(\|u(t)\|_{s+\frac12}+\|\dt u(t)\|_{s-\frac12}\Bigr).
			\]
			Conversely, if $(q,\qb)\in C^0\bigl([0,T],H_0^s(\T^n,\text{c.c.})\bigr)$
			is a solution of \eqref{miserve44} with
			\begin{equation}\label{aaaa}
			\max_{t\in[0,T]}\|q(t)\|_{m_0}\leq\delta,
			\end{equation} then $(u,v):=\Phi(q,\qb)$
			satisfies $v=\dt u$ with
			\[
			u\in C^0\bigl([0,T],H_0^{s+\frac12}(\T^n,\R)\bigr)\cap
			C^1\bigl([0,T],H_0^{s-\frac12}(\T^n,\R)\bigr).
			\]
			Moreover, $u$ is a solution of \eqref{MK} with the following estimate
			\begin{equation}\label{ffff}
			\max_{t\in[0,T]}\Bigl(\|u(t)\|_{s+\frac12}+\|\dt u(t)\|_{s-\frac12}\Bigr)
			\leq\mathfrak{C}_0\max_{t\in[0,T]}\|q(t)\|_s.
			\end{equation}
	\end{proposition}
\begin{proof}
	{\bf Item 1:}\\ By Lemma \ref{lemma1}, \ref{contract1} and \ref{contract2}, the
	maps $\Phi^{(2)},\Phi^{(3)},\Phi^{(4)}$ are invertible with continuous inverse
	on the balls $B_{\delta_0}$, $B_{\delta_2}$, $B_{\delta_4}$ respectively, and
	each satisfies a bound of the form $\|\cdot\|_s\leq \mathfrak{C}\|\cdot\|_s$, with the constant $\mathfrak{C}$ depending only on $c$.\\  By choosing $\delta$ small enough
	that \eqref{miserve33}, the composition is well defined and invertible.\\ Finally, the bounds \eqref{miserve22} and
	\eqref{miserve22bis} follow by composing the corresponding estimates \eqref{bond1},\eqref{bond2},\eqref{bond3},\eqref{bond4} with
	$\mathfrak{C}_0$ suitably chosen.\\
	{\bf Item 2:}\\ For every fixed $s\geq m_0$, both $\Phi$ and $\Phi^{-1}$ are
	Lipschitz maps between the corresponding phase spaces, hence they map curves
	$$(q(t),\qb(t)):[0,T]\rightarrow C^0([0,T],H^s(\T^n,c.c.)\cap C^1([0,T],H^{s-1}(\T^n,c.c.))$$  into curves, $(u(t),v(t))$, with the same time interval. \\
	Finally, by imposing the bounds \eqref{bbbb} and \eqref{aaaa} we can use  \eqref{miserve22} and \eqref{miserve22bis} respectively on $(q(t),\qb(t))$ and $(u(t),v(t))$ for every $t\in[0,T]$.
\end{proof}
%
%
	\subsection{Proof of Theorem \ref{Theorem1}}
	Let us consider $\Phi$ defined in \eqref{Phi} and $\delta$ in Proposition \ref{prop:Phi:global} and set $r$ such that $r\leq\delta$.\\
	By \eqref{change}, Lemma \ref{lemmablock}, Lemma \ref{lemmafirst} and Lemma \ref{lemmasecond} the transformation $\Phi$ conjugates system
	\eqref{system1} to
	\begin{equation*}\label{final:field}
	\partial_t\begin{pmatrix} q\\ \bar q\end{pmatrix}
	=X^{(2)}(q,\bar q)
	=\Bigl(1+\P\bigl(\Phi^{(4)}(q,\bar q)\bigr)\Bigr)
	\bigl(D_1(q,\bar q)+X_{\mathrm{Res},3}(q,\bar q)\bigr)
	+X_{\mathrm{Res},5}(q,\bar q)+X^{(2)}_{\geq 7}(q,\bar q),
	\end{equation*}
	hence the decomposition \eqref{Zsplitting} in the statement holds with
	\begin{equation}\label{deff}
	\mathcal{D}_1:=D_1,\quad\quad\mathcal{Z}_3:=X_{{\bf Res},3},
	\qquad
	\mathcal{Z}_5:=X_{{\bf Res},5},
	\qquad\F:=\P\circ\Phi^{(4)},\qquad
	\mathcal{R}_{\geq7}:=X^{(2)}_{\geq7}.
	\end{equation}

	The vector fields $\mathcal{Z}_3$ and $\mathcal{Z}_5$ contain only terms of homogeneity three and
	five respectively, by \eqref{X311}-\eqref{X312} and
	\eqref{X51}-\eqref{X521}.\\
	 Moreover,, a direct
	computation then gives
	\[
	[\mathcal{D}_1,\mathcal{Z}_3]=0,\qquad [\mathcal{D}_1,\mathcal{Z}_5]=0 .
	\]

	Let now $(q,\bar q)$ be a solution of the truncated system
	\begin{equation}\label{truncated:system}
		\dt(q,\qb)=\F\,\mathcal{D}_1+\bigl(1+\F\bigr)X_{{\bf Res},3}+X_{{\bf Res},5}.
	\end{equation}
	 By \eqref{cancellazione} the resonant cubic terms give
	no contribution to the energy estimate, and by \eqref{cancellazione2} the same
	holds for the resonant quintic ones.\\ 
	Finally since $\mathcal{D}_1$ is skew-adjoint with respect
	to $\langle\Lambda^{2s}\cdot,\cdot\rangle$ and $\mathcal{F}$ is a real scalar function of
	time only by \eqref{P} and \eqref{deff}, we obtain
	\[
	\partial_t\|q\|_s^2=0,\qquad s\geq m_0,
	\]
	that is, the Sobolev norms of the solutions of \eqref{truncated:system} are
	constant.
	
	Finally, $R_{\geq7}=X^{(2)}_{\geq7}$ contains only terms of homogeneity at least
	seven and, by \eqref{stimX2geq7}, it maps $H^s_0(\T^n,\mathrm{c.c.})$ into
	itself for every $s\geq m_0$.\\
	This concludes the proof of Theorem \ref{Theorem1}.
	\qed
	\subsection{Proof of Theorem \ref{theorem!}}

	Let $(u_0,v_0)\in H_0^{s+\frac12}(\T^n,\R)\times H_0^{s-\frac12}(\T^n,\R)$ with
	$s\in[\frac32,2)$ satisfying
	\begin{equation}\label{delk}
	\varepsilon:=\|u_0\|_{m_0+\frac12}+\|v_0\|_{m_0-\frac12}\leq\varepsilon_0
	:=\frac{\delta}{2\mathfrak{C}_0},
	\end{equation}
	where $\delta$ is the one of Proposition \ref{prop:Phi:global}.\\
		By the classical local well-posedness theory for the Kirchhoff equation
	(see \cite{Dickey} or \cite{ArosioPanizzi}), equation \eqref{MK} admits a unique
	local solution with initial data $(u_0,v_0)$.\\
	By Proposition
	\ref{prop:Phi:global} this solution corresponds, as long it remains in ball $B_\delta$, to the unique solution $(q(t),\qb(t))$ of \eqref{miserve44} with
	initial data $(q_0,\qb_0):=\Phi^{-1}(u_0,v_0)$.\\
	By \eqref{miserve22}, we have
	\begin{equation}\label{tttt}
	\|q_0\|_{m_0}\leq\mathfrak{C}_0\,\varepsilon\leq\frac{\delta}{2},
	\end{equation}
moreover,by using the energy estimate \eqref{ema1} we have
	$$\dt\|q\|_{m_0}^2\leq \c_2\|q\|_{m_0}^8.$$
	From this, we get
	\begin{equation}\label{gggg}\|q(t)\|_{m_0}\leq\dfrac{\|q_0\|_{m_0}}{(1-3\c_2\|q_0\|^6_{m_0})^{\frac{1}{6}}}\leq2\|q_0\|_{m_0},
		\end{equation}
for a given $T\sim \epsilon^{-6}$ and every $t\in[0,T]$.\\
	At this point, the bound \eqref{miserve11} for $u\in H^{m_0}$ is then implied by \eqref{ffff} , \eqref{delk} and \eqref{gggg}, for $C=2\mathfrak{C}_0$.\\
	If, in addition $(u_0,v_0)\in H^{s+\frac{1}{2}}\times H^{s-\frac{1}{2}}$ for some $s\geq m_0$ then the energy estimate \eqref{ema1} together with a Gronwall argument provide us the corresponding bound \eqref{miserve11}.\hfill $\Box$

\newpage
	\bibliographystyle{plain}
	\bibliography{BibliografiaArticoletto1}
\end{document}